\documentclass[11pt]{amsart}
\usepackage{amssymb}
\usepackage{amsfonts}
\usepackage{amscd}
\usepackage{amsmath}
\usepackage{amsthm}
\usepackage{mathrsfs}
\usepackage{curves}
\usepackage{epsfig}
\usepackage[pass]{geometry}
\usepackage{graphicx}
\usepackage{verbatim}
\usepackage{latexsym}
\usepackage{eucal}
\usepackage{bbm}
\usepackage{tikz}
\usetikzlibrary{matrix}

\numberwithin{equation}{section}

\newtheorem{theorem}{Theorem}[section]
\newtheorem{lemma}[theorem]{Lemma}
\newtheorem{prop}[theorem]{Proposition}
\newtheorem{cor}[theorem]{Corollary}

\newtheorem{definition}[theorem]{Definition}

\newtheorem{remark}[theorem]{Remark}

\newtheorem*{lemma*}{Lemma}

\def \mca {{\mathscr A}}
\def \mcb {{\mathscr B}}
\def \mcc {{\mathscr C}}
\def \mcd {{\mathscr D}}

\def \mci {{\mathscr I}}

\def \mcl {{\mathscr L}}
\def \mcm {{\mathscr M}}

\def \mcr {{\mathscr R}}

\def \mcv {{\mathscr V}}
\def \mcw {{\mathscr W}}

\def \mbr {{\mathbb R}}
\def \mbs {{\mathbb S}}

\def \loc {\text{loc}}

\def \defeq {\stackrel{\operatorname{def}}{=}}
\def \beqq {\begin{equation}}
\def \eeqq {\end{equation}}

\def \bpf {\begin{proof}}
\def \epf {\end{proof}}
\def \beq {\begin{equation*}}
\def \eeq {\end{equation*}}

\def \eps {\epsilon}   
   
\def \La {\Lambda}

\def \p {\partial}
\def \ha {\frac{1}{2}}

\def \tilde {\widetilde}

\begin{document}
\title{Streak artifacts from non-convex metal objects in X-ray tomography}
\author{Yiran Wang}
\address{Yiran Wang
\newline
\indent Department of Mathematics, Emory University}
\email{yiran.wang@emory.edu}
\author{Yuzhou Zou}
\address{Yuzhou Zou
\newline
\indent Department of Mathematics, Stanford University}
\email{zou91@stanford.edu}
\begin{abstract}
We study artifacts in the reconstruction of X-ray tomography due to nonlinear effects. For non-convex metal objects, we analyze the new phenomena of streak artifacts from inflection points on the boundary of metal objects.  We characterize the location and strength of all possible artifacts using notions of conormal distributions associated with the proper geometry.  
\end{abstract}
\date{\today}
\maketitle

\section{Introduction}\label{sec-intro}
Consider metal artifacts in CT scan, see for example \cite{Seo0} for the background. We begin with the mathematical setup in \cite{Seo} for beam hardening effects. Later, we will work with a more general setup. 
Let $f_E(x)$ be the attenuation coefficients which depends on $E$ the energy level. We assume that for $E \in [E_0 -\eps, E_0 + \eps]$, $\eps > 0$, 
\beq
f_E(x) = f_{E_0}(x) + (E - E_0)f \quad\text{where } f = \alpha\chi_D,
\eeq
with $\chi_D$ the characteristic function for a metal region $D\subset \mbr^2$, and $\alpha > 0$ a constant which approximates $\p f_E/\p E$ over $D$.  Let $\mcr$ denote the Radon transform on $\mbr^2$. The X-ray data, also called sinogram, is given by 
\beqq\label{eq-ctp}
P = \mcr f_{E_0} + P_{MA}, \ \ P_{MA} = - \ln\left(\frac{\sinh(\eps\mcr f)}{\eps\mcr f}\right).
\eeqq
The term $P_{MA}$ is derived from the Beer-Lambert law under some assumptions, see \cite{Seo}. We emphasize that the existence of the term is due to the dependency of $f_E$ on the energy level $E$. If $\alpha = 0$ so that $f = 0$, it is clear that $P_{MA} = 0$ and $P$ is exactly the Radon transform $\mcr f_{E_0}$, which is commonly assumed in CT scan. One can apply the filtered backprojection (FBP) to get $f_{E_0}$, namely 
\beq
f_{E_0} =  \mcr^*\mci^{-1} \mcr f_{E_0}.
\eeq
Here, $\mci^{-1}$ is the Riesz transform and $\mcr^*$ denotes the adjoint of $\mcr$. 
For $P$ in \eqref{eq-ctp}, we get 
\[
 f_{CT} = \mcr^*\mci^{-1} P = f_{E_0}  +  f_{MA}, \ \ f_{MA} = \mcr^*\mci^{-1} P_{MA}.
\] 
Note that we can write
\[f_{MA} = \mcr^*\mci^{-1}F(\mcr f),\quad\text{where }F(x) = -\ln\left(\frac{\sinh(\eps x)}{\eps x}\right)\text{ is smooth}.\]
The term $f_{MA}$ often causes streak artifacts  in the reconstruction, see Figure \ref{fig-art} for an illustration. An outstanding problem is to understand the mechanism of the artifact generation  and alleviate the effects. 
Our goal of this work is to give a quantitative description of the possible artifacts. We state a consequence of our main result. 
\begin{theorem}\label{thm-main}
Suppose $D$ is a simply connected bounded open domain in $\mbr^2$ with smooth boundary $\p D$ and satisfies assumptions (A1), (A2) in Section \ref{sec-lag}. Then 
\beq
\text{sing supp}(f_{MA})\subset \left(\bigcup_{L\in\mcl}{L}\right) \cup \p D.
\eeq 
where $\mcl$ is the collection of lines  $L$ which is either tangent at two non-inflection points, or tangent at only one inflection point, see Figure \ref{fig-art}. 
\end{theorem}
(A1), (A2) are assumptions on the geometry of $D$ to simplify some analysis. We recall that $p$ is an infection point on $\p D$ if the curvature $\kappa = 0$ at $p$ and changes sign across $p$. The method we use also allows us to find the strength of the artifacts. The precise statement, Theorem \ref{thm-main1}, is stated in Section \ref{sec-lag} after  (A1), (A2)  and proper conormal distribution spaces are introduced.


\begin{figure}[htbp]
\centering
\includegraphics[scale = 0.17]{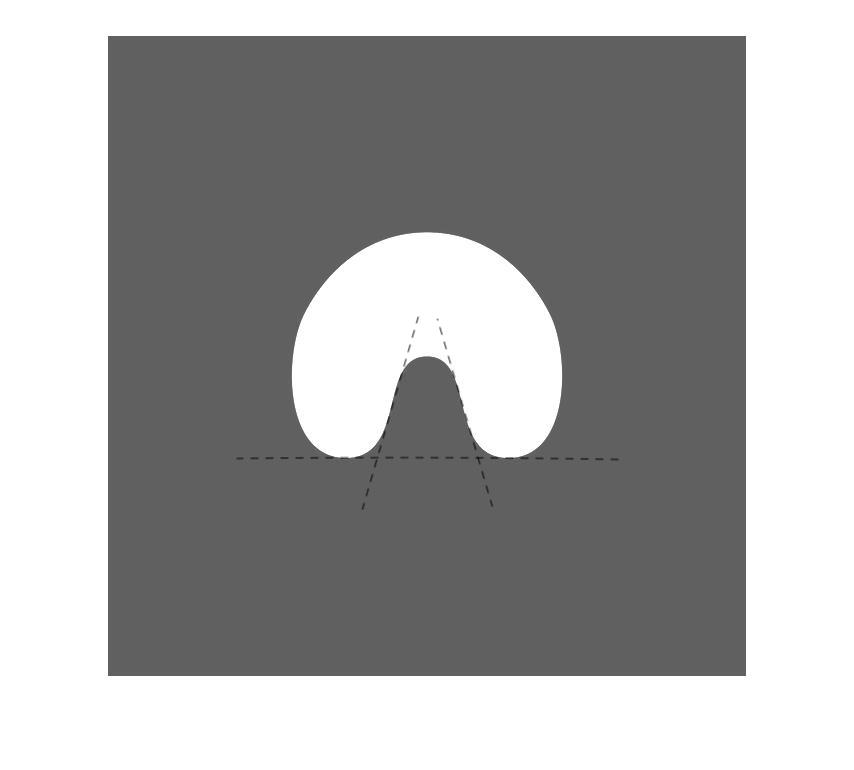} 
\includegraphics[scale = 0.17]{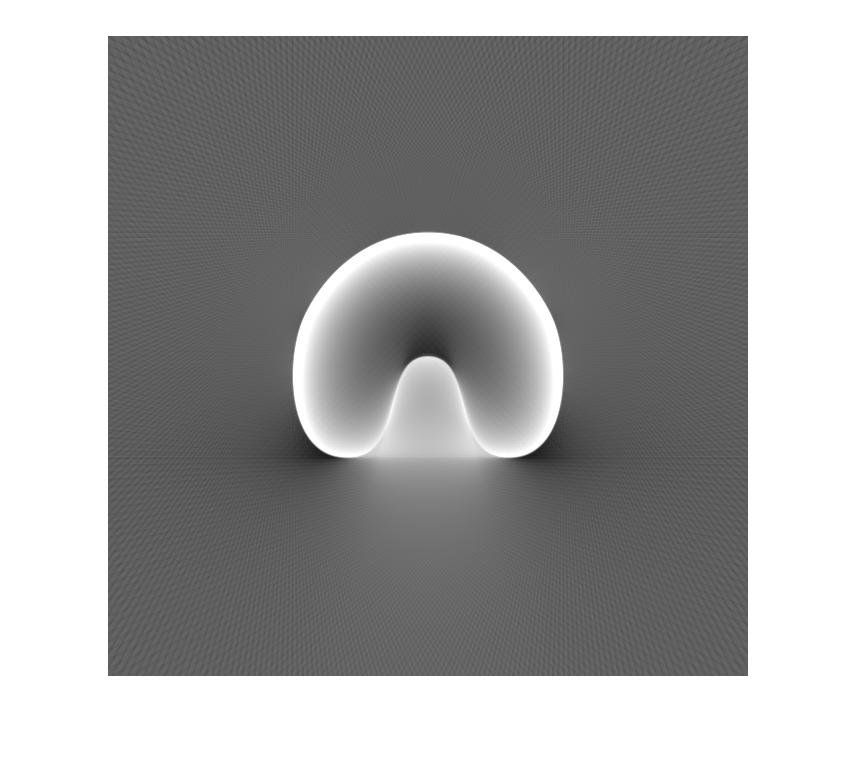} 
\includegraphics[scale = 0.34]{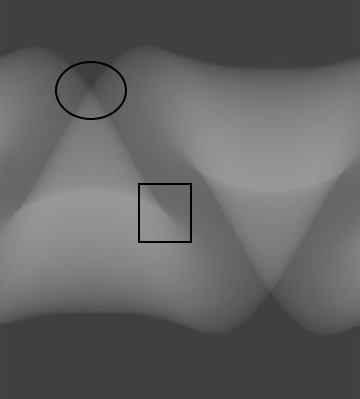}
\caption{Numerical simulation of steak artifacts for a non-strictly convex metal object. {\bf Left figure:} The white region is the non-convex metal object $D$. We marked the location of the artifacts predicted by Theorem \ref{thm-main} by the dashed lines. There are two lines corresponding to the two inflections points and one line that is tangent to two boundary points. {\bf Middle figure:} The reconstructed image using FBP with the nonlinear function $F(u) = a u^2$ for some constant $a$.  {\bf Right figure:} The sinogram of $\mcr\chi_D$. We circled the transversal intersection point in the oval and this corresponds to the artifact tangent to two boundary points. We circled the cusp point   in the square and this corresponds to the artifacts at the inflection point. }
\label{fig-art}
\end{figure}

The mathematical study of metal artifacts started from Park-Choi-Seo \cite{Seo} in which the authors characterized the artifacts using the concept of wave front set. It turns out that the singularity associated with the artifacts is due to nonlinear interactions of the singularities in $\mcr\chi_D$, which is intimately related to the geometry of $D$. For example, for strictly convex objects, the artifacts are straight lines tangent to two boundary curves, see the line $L$ in Figure \ref{fig-art}. Using more precise notions of conormal distributions and paired Lagrangian distributions, Palacios, Uhlmann and Wang in \cite{PUW1} gave a quantitative analysis of the metal artifacts. Moreover, metal regions with piecewise strictly convex boundaries are addressed and the strength of the artifacts were obtained in \cite{PUW1}. We also mention the work \cite{BFJQ} in which artifacts from incomplete X-ray data are analyzed from the microlocal point of view.


In this paper, our goal is to study general metal object, especially non-strictly convex ones. 
When $\p D$ contains inflection points, we show that the sinogram  contains cusp points, see Figure \ref{fig-trans}. 
We show that the nonlinear interaction of cusp singularities generates new ones that lead  to the artifacts. Another motivation is that we would like to relax the regularity assumption in previous work \cite{PUW1}. While keeping $\chi_D$ in mind which is a classical conormal distribution, we replace it by $H^s$ based conormal distribution of finite orders, see Section \ref{sec-lag} for details. This allows us to use some techniques that originated in the study of  singularities of solutions of nonlinear wave equations, see Melrose-Ritter \cite{MR1, MR2}.  We believe that the method we develop will be helpful for understanding artifacts in other tomography methods, for example  the attenuated X-ray transform arising from SPECT, see Katsevich \cite{Kat}. 

\begin{figure}[htbp]
\centering
\includegraphics[scale =0.6]{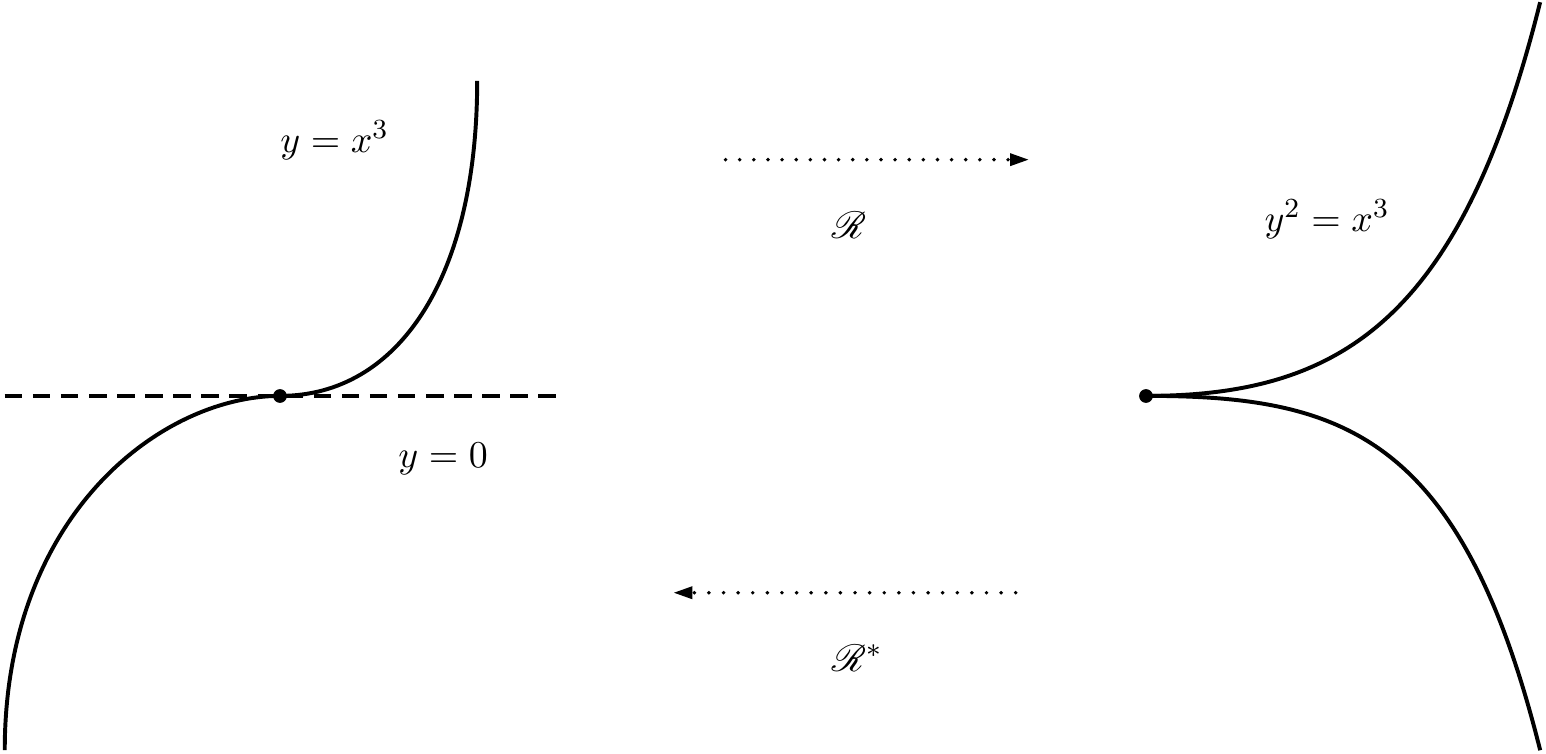}
\caption{Illustration of the transformation of inflection points in $\mbr^2$ to cusp points in sinogram. The nonlinear interaction produces singularities at the cusp point which transforms back to the tangent line at the inflection point.}
\label{fig-trans}
\end{figure}

The paper is organized as follows. We state the assumptions and main result of the paper in Section \ref{sec-lag}. 
In Section \ref{sec-cusp0}, we discuss some microlocal aspects of the Radon transform and study the connection of  inflection points and cusp points in sinogram. In Section \ref{sec-commute}, we prove some technical lemmas to help translate iterated regularity results between the physical and sinogram spaces via $\mcr$. We    discuss the nonlinear interactions in Section \ref{sec-non} and  finish the proof in Section \ref{sec-char}. Finally, we investigate the generation of new singularities at cusp points in Section \ref{sec-new}.

\section{Statement of the main result}\label{sec-lag}
In this work, we consider a simply connected bounded domain $D\subset \mbr^2$ with smooth boundary $\gamma \defeq \p D$. The case of multiple domains can be addressed similarly. The boundary $\gamma$ is a simple closed plane curve, see for instance \cite{Kl}. This means that $\gamma$ is a smooth parametrized curve ${\gamma: [a, b] \rightarrow \mbr^2}$ with no self-intersections such that $\gamma$ and all its derivatives agree at $a, b$.

We shall consider attenuation coefficients that have conormal type singularities to $\gamma$. So 
we start introducing the notion of such distributions, following Melrose-Ritter \cite{MR1}. Let $Y$ be a smooth embedded submanifold of a smooth  manifold $X$. We denote by $\mcv(Y)$ the Lie algebra of smooth vector fields tangent to $Y$. The space of $H^s$-based conormal distributions of order $k$ is defined as 
 \beqq\label{eq-IkHs}
 I_k H^s(X; \mcv(Y)) = \{u\in H^s(X): V_1V_2\cdots V_j u \in H^s(X), \forall j \leq k, V_j \in \mcv(Y)\}.
 \eeqq
We also use $L^2$-based conormal distributions and denote  $I_k H^0(X; \mcv(Y)) =   I_k L^2(X; \mcv(Y)) $. The set of  bounded elements is denoted by $L^\infty I_k L^2(X; \mcv(Y)) = L^\infty \cap  I_k L^2(X; \mcv(Y))$. These spaces were initially introduced  and are still widely used in the analysis of singularities of solutions to nonlinear wave equations, see Melrose-Ritter \cite{MR1, MR2}. In particular, the Gagliardo-Nirenberg inequalities of \cite{MR1} imply that $L^\infty I_k L^2(X; \mcv(Y))$ is a $C^\infty$ algebra, which means that if $u\in L^\infty I_k L^2(X; \mcv(Y))$ and $F \in C^\infty$, then $F(u) \in L^\infty I_k L^2(X; \mcv(Y)). $ We remark that when $k = \infty$, the space $I_\infty H^s(X; \mcv(Y))$ (i.e. $\cap_{k\in\mathbb{N}}{I_kH^s(X;\mcv(Y))}$)  consists of  classical conormal distributions, see H\"ormander \cite[Definition 18.2.6]{Ho3}, except that to match the order one should use Besov spaces instead of Sobolev spaces.

In the context of Theorem \ref{thm-main} we consider $f = \alpha\chi_D$, which is a classical conormal distribution; more generally we work with $f\in L^\infty I_k L^2(\mbr^2; \mcv(\gamma))$  with compact support.   We will show that for such $f$ the metal artifact $f_{MA}$ is a distribution conormal to $\gamma$ and certain tangent lines.  To state a clear result, we make the following assumptions on the geometry of $D$.  There is essentially no loss of generality. Our method would apply to the general case with some modifications which we remark on later in the proof. The assumptions are
\begin{enumerate}
\item[(A1)] The curvature $\kappa$ on $\gamma$ does not vanish on open sets of $\gamma.$ When $\kappa = 0$ at $p\in \gamma$, $p$ is a simple inflection point, see Section \ref{sec-cusp}. 
\item[(A2)] For any straight line $L$ tangent to $\p D$, $L$ is tangent to $\gamma$ at a finite order and $L$ is either tangent at at most two non-inflection points, or tangent at only one inflection point.  
\end{enumerate} 
We denote by $\mcl$ the (finite) set of tangent lines $L$ which are tangent either at two non-inflection points or one inflection point. 

The notion of $I_kH^s$ in \eqref{eq-IkHs} can be generalized to a collection of submanifolds. 
Suppose $S = \{Y_1,\dots,Y_n\}$ is a collection of smooth submanifolds of $X$. Let $\mcv(S) = \mcv(Y_1,\dots,Y_n)$ denote the Lie algebra of smooth vector fields tangent to each submanifold in $S$. Then we can define $I_kH^s(X;\mcv(S))$ similar to \eqref{eq-IkHs}:
 \beqq
 I_k H^s(X; \mcv(S)) = \{u\in H^s(X): V_1V_2\cdots V_j u \in H^s(X), \forall j \leq k, V_j \in \mcv(S)\}.
 \eeqq
Note that
\[u\in I_{\infty}H^s(X;\mcv(S))\implies \text{sing supp}(u)\subset\cup_{Y\in S}{Y}.\]
Associated with the tangent lines $L$, we introduce 
\beq
\mcd_k^s(\mcl)  = \sum_{L \in \mcl }  I_kH^s(\mbr^2; \mcv(L, \gamma)). 
\eeq
This is a space based on the conormal distribution space associated with $L$ and $\gamma$ among certain lines $L$ which intersect $\gamma$ tangentially, namely the lines in $\mcl$. 

The main result of the paper is 
\begin{theorem}\label{thm-main1}
Suppose $D$ is a simply connected bounded open domain in $\mbr^2$ with smooth boundary $\p D$ and satisfies assumptions (A1), (A2). Suppose $f\in L^\infty I_k L^2(\mbr^2; \mcv(\gamma))$ is compactly supported.  Then for any smooth function $F: \mbr\rightarrow \mbr$, we have
\beq
\mcr^*\mci^{-1} F(\mcr f) \in \mcd_k^{-\ha}(\mcl).
\eeq
In particular, 
away from $\gamma$, we have 
\beq
\mcr^*\mci^{-1} F(\mcr f) \in I_k H^{-\ha}(\mbr^2; \mcv(\mcl)).
\eeq
\end{theorem}
To obtain Theorem \ref{thm-main}, we notice that $f_{MA} = \mcr^*\mci^{-1} F(\mcr f)$ for $F(x) = -\ln(\sinh(\eps x)/\eps x)$ and $f = \alpha\chi_D \in I_{\infty}L^2(\mbr^2; \mcv(\gamma))$. 
It follows that, away from $\gamma,$ we have $f_{MA} \in I_\infty H^{-\ha}(\mbr; \mcv(\mcl))$. Hence the singular support of $f_{MA}$ away from $\gamma$ is contained in $\cup_{L\in\mcl}{L}$, and thus overall we have $\text{sing supp}(f_{MA})\subset(\cup_{L\in\mcl}{L})\cup\gamma$.

Note that the loss of $\ha$ order is due to using the $C^{\infty}$-algebra properties of $L^{\infty}I_kL^2$ and can likely be improved; see Remark \ref{sob-order}. On a related topic, it is worth mentioning that the metal artifact reduction method proposed in \cite{PUW1} should apply to the current setting. However, to justify the reduction effect, one needs to work harder to find the precise order of the artifacts which we do not pursue here.

  \section{Generation of cusp points from inflection points}\label{sec-cusp0}
	
In this section, we review some basic properties about the Radon transform, including its canonical relation, and apply the result to see how a cusp point in the sinogram arises from an inflection point in physical space. 
We recall that the Radon transform is usually defined as an operator $\mcr: C_0^\infty(\mbr^2)\rightarrow C^\infty(M)$ with $M \defeq \mathbb{R}\times\mathbb{S}^1$ by
\[
\mcr f(s, \theta) = \int_{x\cdot \theta = s} f(x)\, d\mathcal{H}^1(x), \quad x \in \mbr^2 
\]
where $d\mathcal{H}^1(x)$ is the 1-dimensional Hausdorff measure on the line $\{x\cdot \theta = s\}$. Here, $s\in \mbr, \theta\in \mbs^1$ and we can parametrize $\mbs^1$ as $\theta = (\cos \phi, \sin \phi), \phi \in (-\pi, \pi]$. From this definition one can derive many well-known properties of the Radon transform; see \cite{Quinto} for examples. 

Under this definition, the Radon transform satisfies the symmetry property $\mcr f(s,\theta) = \mcr f(-s,-\theta)$ stemming from the fact that the lines $\{x\cdot\theta = s\}$ and $\{x\cdot(-\theta) = -s\}$ are the same line. This suggests that one can instead directly define the Radon transform to map into the space of functions on the space $\mathcal{L}$ of lines in $\mbr^2$, topologized appropriately. In fact, for later arguments in Section \ref{sec-commute} (especially regarding the existence of parametrices) it will be more convenient for us to define $\mcr$ \emph{directly on this space of lines $\mathcal{L}$} (which we also call the ``sinogram space''), which will be a quotient of $M$. We discuss this viewpoint in Section \ref{sec-radon} and review the well-known properties of the Radon transform once redefined with respect to $\mathcal{L}$ before showing how this connects inflection points and cusps in Section \ref{sec-cusp}.
	
 \subsection{Radon transform preliminaries}\label{sec-radon} 

We begin by considering the set $\mathcal{L}$ of lines in $\mbr^2$. For each line $L\in\mathcal{L}$ we can write it in the form $L=\{x\in\mbr^2\,:\,x\cdot\theta = s\}$ for some $s\in\mathbb{R}$ and $\theta\in\mathbb{S}^1$; thus $\theta$ is a unit normal vector to the line, and $s$ is the (signed) distance from the origin with respect to this choice of $\theta$. Since $(s,\theta)$ and $(-s,-\theta)$ parametrize the same line, we can identify $\mathcal{L} = (\mathbb{R}\times\mathbb{S}^1)/\sim$ where $(s,\theta)\sim(-s,-\theta)$. Since $\mathbb{S}^1$ is itself a quotient, we can also identify this with $(\mathbb{R}_s\times\mathbb{R}_{\phi})/\sim_{\mathcal{L}}$ (with $\phi$ denoting the angle), where $(s,\phi)\sim_{\mathcal{L}}(s',\phi')$ iff $\phi'-\phi\in\pi\mathbb{Z}$ and $s' = (-1)^{(\phi'-\phi)/\pi}s$. Note that the projection $p:\mathbb{R}^2\rightarrow\mathcal{L}$ is a covering map, and the corresponding deck transformations are given by $f_k(s,\phi)=((-1)^ks,\phi+k\pi)$ where $k\in\mathbb{Z}$. Note that
\[(f_k)^*\left((\sigma\,ds+\eta\,d\phi)|_{f_k(s,\phi)}\right) = ((-1)^k\sigma\,ds+\eta\,d\phi)|_{(s,\phi)},\]
and for every $[(s,\phi)]\in\mathcal{L}$ the fiber of $T_{[(s,\phi)]}^*\mathcal{L}$ is isomorphic via pullback by $p$ to the fiber $T_{(s,\phi)}^*\mathbb{R}^2$ of any preimage $(s,\phi)$, with the isomorphisms consistent with the pullback property of the deck transformations above. It follows that we can identify $T^*\mathcal{L} = T^*\mathbb{R}^2/\sim_{T^*\mathcal{L}}$, where if we represent $(\sigma\,ds+\eta\,d\phi)|_{(s,\phi)}\in T^*\mathbb{R}^2$ by the coordinates $(s,\phi,\sigma,\eta)$, we have $(s,\phi,\sigma,\eta)\sim_{T^*\mathcal{L}}(s',\phi',\sigma',\eta')$ iff $(s,\phi)\sim_{\mathcal{L}}(s',\phi')$, $\sigma = (-1)^{(\phi'-\phi)/\pi}\sigma'$, and $\eta=\eta'$. Given these identifications, we will use $(s,\phi)$ and $(s,\phi,\sigma,\eta)$ as (local) coordinates for $\mathcal{L}$ and $T^*\mathcal{L}$, respectively.


Having established the space of lines $\mathcal{L}$, we now redefine the Radon transform with respect to it:
\begin{definition}
The Radon transform is the operator $\mcr :C_c^{\infty}(\mathbb{R}^2)\rightarrow C_c^{\infty}(\mathcal{L})$  defined by 
\[\mcr f(L) = \int_{L}{f(x)\,d\mathcal{H}^1(x)}.\] 
\end{definition}
This is essentially the same definition as the definition at the beginning of this section, except the operator now maps into the space of functions on $\mathcal{L}$ instead of functions on $M$.

With respect to the coordinates $(s,\phi)$ on $\mathcal{L}$ and $x$ on $\mathbb{R}^2$, its Schwartz kernel can be written as
\[K(s,\phi,x) = \delta(s-x\cdot \theta_{\phi}) = (2\pi)^{-1}\int_{\mathbb{R}}{e^{i(s-x\cdot \theta_{\phi})\lambda}\,d\lambda},\]
where $\theta_{\phi} = (\cos\phi,\sin\phi)$. Viewed as an operator to functions on $\mathbb{R}^2$, this is a Fourier Integral Operator associated to the 
 canonical relation
\begin{equation}
\label{can-rel} 
\begin{aligned}
C &= \{((s,\phi;\lambda,-\lambda x\cdot \theta_{\phi}^{\perp}),(x;\lambda \theta_{\phi}))\,:\,x\in\mathbb{R}^2, s = x\cdot \theta_{\phi},\lambda\in\mathbb{R}\backslash\{0\}\} \\
&=\{((s,\phi;\sigma,\eta),(x;\xi))\,:\,s = x\cdot \theta_{\phi}, \eta = -\sigma x\cdot \theta_{\phi}^{\perp},\xi = \sigma \theta_{\phi},\sigma\in\mathbb{R}\backslash\{0\}\} \\
&=\{((x\cdot \theta_{\phi},\phi;\sigma,-\sigma x\cdot\theta_{\phi}^{\perp}),(x;\sigma \theta_{\phi}))\,:\,x\in\mathbb{R}^2,\phi\in\mathbb{R},\sigma\in\mathbb{R}\backslash\{0\}\}
\end{aligned}
\end{equation}
where $\theta_{\phi}^{\perp} = (-\sin\phi,\cos\phi)$. Note that this is invariant under the deck transformation pullbacks $(f_k)^*$ (acting on the variables $(s,\phi;\sigma,\eta)$) and can thus indeed be identified as a subset of $T^*(\mathcal{L}\times\mbr^2)$. In H\"ormander's notation, $\mcr $ is an Fourier integral operator from $\mbr^2$ to $\mathcal{L}$, denoted by $\mcr \in {I^{-\ha}(\mathcal{L}\times \mbr^2, C)}$.

Furthermore, we have:
\begin{lemma}
The left and right projections $\pi_L: C\rightarrow T^*\mathcal{L}\backslash 0$ and $\pi_R: C\rightarrow T^*\mathbb{R}^2\backslash 0$ are bijective.
\end{lemma}
This should be compared to the usual situation (see e.g. \cite{Quinto}) where $\pi_R$ is two-to-one (essentially due to the two-to-one identification of $M$ to $\mathcal{L}$).
\begin{proof}
Given $(s,\phi;\sigma,\eta)$ with $(\sigma,\eta)\ne(0,0)$, if $(x;\xi)$ satisfies $((s,\phi;\sigma,\eta),(x;\xi))\in C$, then necessarily $\sigma\ne 0$ (otherwise $\eta = -\sigma x\cdot \theta_{\phi}^{\perp}=0$). 
We must also have $\xi = \sigma\theta_{\phi}$, and the two equations $s = x\cdot \theta_{\phi}$ and $\eta = -\sigma x\cdot \theta_{\phi}^{\perp}$ combine to force $x = s\theta_{\phi}-\frac{\eta}{\sigma}\theta_{\phi}^{\perp}$. Thus, $(x;\xi)$ is uniquely determined by $(s,\phi;\sigma,\eta)$ if $((s,\phi;\sigma,\eta),(x;\xi))\in C$. 
Conversely, given $(x;\xi)$ with $\xi\ne 0$, from the equation $\xi=\sigma\theta_{\phi}$ we must have $|\sigma| = |\xi|$, and hence $\sigma = \pm|\xi|\implies \theta_{\phi} = \pm\frac{\xi}{|\xi|}$. For each choice of sign this determines $\phi$ up to $2\pi\mathbb{Z}$; moreover this also gives $s = \pm\frac{x\cdot\xi}{|\xi|}$ and $\eta = -x\cdot\xi^{\perp}$, where $\xi^{\perp} = (-\xi_2,\xi_1)$. The ambiguities in obtaining these solutions (i.e. choice of sign of $\sigma$ and choice of $\phi$) are identified under the equivalence relation $\sim_{T^*\mathcal{L}}$, i.e. all of these solutions (in $T^*\mathbb{R}^2$) correspond to the same element in $T^*\mathcal{L}$.
\end{proof}

The above injectivity argument also shows that the canonical relation can be viewed as a bijective function $C:T^*\mathbb{R}^2\backslash 0\rightarrow T^*\mathcal{L}\backslash\{\sigma = 0\}$, with
\[C(x,\xi) = \left[\left(\frac{x\cdot\xi}{|\xi|},\text{arg }\xi,|\xi|,-x\cdot\xi^{\perp}\right)\right]\]
where $\text{arg }\xi$ denotes the angle corresponding to the nonzero vector $\xi\in\mbr^2$, and
\[C^{-1}([s,\phi,\sigma,\eta]) = \left(s\cdot \theta_{\phi} - \frac{\eta}{\sigma}\theta_{\phi}^{\perp}, \sigma \theta_{\phi}\right).\]
The above argument also implies that
\[C^*\circ C\subset\Delta_{T^*\mathbb{R}^2\backslash 0}\subset(T^*\mathbb{R}^2\backslash 0)^2\quad\text{and}\quad C\circ C^*\subset\Delta_{T^*\mathcal{L}\backslash 0}\subset(T^*\mathcal{L}\backslash 0)^2\]
where $\Delta$ denotes the diagonal and $C^*$ is the transpose relation of $ C$ (corresponding to the inverse map $C^{-1}$ above).

The transpose relation will come up in considering an adjoint for $\mcr$. On $\mathcal{L}$ we can give a density $|ds\wedge d\phi|$ (possibly normalized say to $\frac{1}{\pi}|ds\wedge d\phi|$), in which case the adjoint $\mathscr{R}^*:C^{\infty}(\mathcal{L})\rightarrow C^{\infty}(\mathbb{R}^2)$ is given up to a multiplicative constant by
\[\mathscr{R}^*g(x) = \int_0^{\pi}{g([x\cdot \theta_\phi,\phi])\,d\phi}.\]
Note that $\mcr^*$ is then an FIO associated to the transpose relation $C^*$. Thus in the double fibration picture, see Guillemin \cite{Gu}, 
 \begin{center} 
\begin{tikzpicture}
  \matrix (m) [matrix of math nodes,row sep=1.5em,column sep=1em,minimum width=1em]
  {
      & C & \\
     T^*\mbr^2 &  & T^*\mathcal{L} \\};
  \path[-stealth]
    (m-1-2) edge node [left] {$\pi_L$} (m-2-1)
    (m-1-2) edge node [right] {$\pi_R$} (m-2-3);
\end{tikzpicture}
\end{center}
we have that the projections $\pi_L$ and $\pi_R$ are injective immersions.
So the double fibration satisfies the Bolker condition as interpreted in  ($A'$)  of \cite{Gu}. 
 The composition $\mcr^*\mcr$ is a pseudo-differential operator of order $-1$. As a result, we obtain from standard $L^2$ estimate of pseudo-differential operators that $\mcr: H^s(\mbr^2)\rightarrow H^{s+ \ha}(\mathcal{L})$ is bounded. 
 
We now consider functions $f$ with conormal singularities to the boundary $\gamma$ and see how the Radon transform transforms the singularities.  
Consider the conormal bundle $N^*\gamma.$ 
Let $\La  = C\circ N^*\gamma$. 
 It is known that if $D$ is a strictly convex domain, 
 then $\La = N^*S$ where $S$ is a one dimensional submanifold of $M$. This can be seen explicitly as follows. Let $\gamma(t) = (x_1(t), x_2(t)), t\in [a, b]$  and $t$ be the arc-length parameter. With respect to the frame determined by the coordinate system,  the curvature is given by $\kappa(t) = \dot x_1(t) \ddot x_2(t) - \dot x_2(t) \ddot x_1(t).$ Note that the curvature is not always non-negative. We recall that $\gamma$ is convex if and only if $\kappa(t)\geq 0$ or $\kappa(t) \leq 0$ on $[a, b]$, see \cite{Kl}. Also, $\gamma$ is strictly convex if $\kappa(t)\neq 0.$ 
 Using the parametrization of $\gamma$, we have  that 
 \beq
 N^*\gamma = \{(x_1(t), x_2(t), \xi_1(t), \xi_2(t)) \in T^*\mbr^2  :  \xi(t)\cdot \dot \gamma(t) = 0\}. 
 \eeq
From the second line of \eqref{can-rel}, we have that if $(s(t),\phi(t),\sigma(t),\eta(t)) = C(x(t),\xi(t))$ with $(x(t),\xi(t))\in N^*\gamma$, then
\[\xi(t) = \sigma(t)\theta_{\phi(t)}\implies \dot x_1(t)\cos\phi(t) + \dot x_2(t)\sin\phi(t) = 0\implies \tan\phi(t) = -\frac{\dot x_1(t)}{\dot x_2(t)}\]
since $\xi(t)\cdot\dot\gamma(t) = 0$. In particular, we have $(\cos\phi(t),\sin\phi(t)) = \pm (-\dot x_2(t),\dot x_1(t))$, as well as the fact that
\begin{equation}
\label{tang-line}
\text{the line }\{x\cdot\theta_{\phi(t)} = s(t)\}\text{ is tangent to }\gamma\text{ at }x(t).
\end{equation}
Furthermore, we have $s(t) = x_1(t)\cos\phi(t) + x_2(t)\sin\phi(t)$. Thus, let $S$ be the curve in $\mathcal{L}$ defined by
 \beq
\label{l-curve}
 S = \{[(s(t), \phi(t))] \in \mathcal{L} : s(t) =  x_1(t)\cos \phi(t) + x_2(t) \sin\phi(t), 
  \tan\phi(t)  = -\dot x_1(t)/\dot x_2(t), t \in [a, b]\}.
 \eeq
(Notice that if we defined $S$ as a subset of $M$ instead, then $S$ would be the union of two disjoint curves, which are identified as the same under $\sim$ in $\mathcal{L}.$ Furthermore, note that this definition remains invariant even if we choose a non-unit-speed parametrization $\gamma(t)$ for the curve.) 

For $(s(t),\phi(t))$ defined in the equations in \eqref{l-curve}, we have
$\dot\phi(t) = -\ddot x_1(t)\dot x_2(t) + \ddot x_2(t) \dot x_1(t) = \kappa(t)$ and $\dot s(t) = (-x_1(t)\sin\phi(t)  + x_2(t) \cos\phi(t) )\kappa(t) $.  If $\kappa(t)  = 0$, we see that $\dot s(t)  = \dot \phi(t)  = 0$ and this is where the curve $S$ may not be smooth. Otherwise, if $\kappa(t)\ne 0$, then $\dot\phi(t)\ne 0$. Thus, if there are no inflection points, then $S$ is a smoothly embedded curve, and in fact if $D$ is strictly convex then $S$ has no self-intersections, since a self-intersection corresponds to a straight line being tangent to $\gamma$ at two different points, which would be impossible for $D$ strictly convex.

We claim that $C\circ N^*\gamma = N^*S \backslash 0$. Indeed, we have
 \beqq
\label{canon-comp}
 \begin{gathered}
 C\circ N^*\gamma = \{[(s(t), \phi(t), \sigma(t), \eta(t))] \in T^*\mathcal{L}\backslash 0 : s(t) =  x_1(t)\cos \phi(t) + x_2(t) \sin\phi(t), \\
 \sigma(t)\theta_{\phi(t)}\cdot\dot\gamma(t) = 0, 
 \eta(t) = \sigma(t)(x_1(t)\sin\phi(t)-x_2(t)\cos\phi(t)), \sigma(t)\in\mathbb{R}\backslash\{0\} \}.
 \end{gathered}
 \eeqq
It is straightforward to verify that $\sigma(t)\dot s(t) + \eta(t)\dot\phi(t) = 0$ in the parametrization above. Since $[(s(t),\phi(t))]$ is a smooth parametrization of $S$, this implies that $[(s(t),\phi(t),\sigma(t),\eta(t))]\in N_{[(s(t),\phi(t)]}^*S$. On the other hand, since $S$ is a codimension one curve, it follows that $N^*S$ has one-dimensional fibers, and hence the above parametrization parametrizes all of $N^*S\backslash 0$ since $\sigma(t)$ can take on arbitrary real nonzero values.
 
 \subsection{Generation of cusp points}\label{sec-cusp}
Consider a non-strictly convex domain $D$. Assume that the boundary  is parametrized by $\gamma(t), t\in [a, b]$.  
Let $\kappa(t)$ be the curvature on $\gamma(t)$. By (A1), there are finitely many inflection points $p_i = \gamma(t_i),  t_i \in [a, b], i = 1, \cdots, N$. 
We split the curve $\gamma$ into disjoint unions of strictly convex curves. Let $\gamma_i = \gamma(t)|_{(t_i, t_{i+1})}, i = 1, 2, \cdots, N$. On each $\gamma_i$, the curvature $\kappa$ is nowhere vanishing. 
 Now we examine what happens at  $p_i$, that is when $\kappa = 0.$ Locally near $p_i$, we choose a parametrization $\gamma(t)$ and make an affine change of coordinates such that $\gamma(0) = p_i$, $x_1(t) = t$ for $t\in (-\delta, \delta)$ with $\delta > 0$ sufficiently small, and $\dot x_2(0) = 0$. Then 
 \beq
 \kappa(t) =  \dot x_1(t) \ddot x_2 (t) - \dot x_2 (t) \ddot x_1(t) =  \ddot x_2(t)
 \eeq
 vanishes at $t = 0$. Thus we can assume that $x_2(t) = t^3 h(t)$ where $h(t)$ is a smooth function on $(-\delta, \delta)$. We write $h(t) = \sum_{n = 0}^\infty h_n t^n$ as a Taylor expansion so that formally 
 \beqq\label{eq-curexp}
 \kappa(t) = \sum_{n = 1}^\infty \kappa_nt^n,\quad \kappa_n=(n+2)(n+1)h_{n-1}\text{ for }n\ge 1.
 \eeqq
 If $p_i$ is an inflection point, then $\kappa(t)$ changes sign across $t = 0.$ We call $p_i$ an inflection point of order $k$ for some $k$ odd if $\kappa_i = 0, i < k$ and $\kappa_k\neq 0$. The simplest case is when $x_2(t) = t^3 h(t), h(0) \neq 0$ (so that $p_i$ is an inflection point of order $1$), and this is called a simple inflection point. We analyze this case in this section but remark that the treatment for higher order cases are similar. 
 
 Without loss of generality, we assume that $h(t) > 0. $ Consider 
\[
\gamma = \{(x_1(t), x_2(t)) \in \mbr^2 : x_1(t) = t, x_2(t) = t^3h(t), h(t) > 0, t\in (-\delta, \delta)\}
\]
We denote by $\gamma_\pm$ the pieces of $\gamma$ where $\pm t > 0.$ Each $\gamma_\pm$ is strictly convex, hence we know that 
$C\circ N^*\gamma_\pm = N^*S_\pm\backslash 0$ and 
 \beq
 \begin{gathered}
 S_\pm = \left\{[(s(t), \phi(t))]\in \mathcal{L}: s =  t\cos \phi + t^3 h(t) \sin\phi, 
  \phi = \arctan\left(-\frac{1}{3t^2 h(t) + t^3 h'(t)}\right), \pm t > 0\right\}
\end{gathered}
 \eeq
Note in the parametrization above that as $\pm t\to 0^+$, we have $s\to 0$ and $\phi\to(\frac{\pi}{2})^+$. We now show:
\begin{lemma}
\label{cusp-sinogram}
$S_+ \cup S_-$ form a cusp at $(s,\phi) = (0,\pi/2)$, and furthermore $C\circ N^*\gamma$ is the closure of $N^*S_+\cup N^*S_-$.
\end{lemma}
\begin{proof}
First, we introduce a new variable to simplify the calculation. Let $w =  t(h(t) + t h'(t)/3)^\ha$. Here, we can shrink $\delta$ so that $h(t) + t h'(t)/3 >  0$ on $(-\delta, \delta)$. By the inverse function theorem, we see that $t = w g(w)$ on some $(-\eps, \eps)$ with $g(w)$ smooth and $g(0) \neq 0$.  

Using $w$, we see that $\tan \phi = -1/(3w^2)$ and that 
\beq
w = (-3 \tan \phi)^{-\ha} \text{ when } t > 0  \text{ and } w = -(-3\tan \phi)^{-\ha} \text{ when } t < 0. 
\eeq
Now we find that  for $t > 0$ (i.e. on $S_+$) we have
\beq
\begin{gathered}
s = t \cos \phi + t^3 h(t) \sin \phi  
= wg(w)\cos\phi + w^3  g^3(w) h(wg(w))\sin\phi \\
= g(w) (1/3)^\ha \frac{(-\cos\phi)^{3/2}}{(\sin\phi)^\ha} +   g^3(w)  h(wg(w)) (1/3)^{\frac{3}{2}}\frac{(-\cos\phi)^{3/2}}{(\sin\phi)^\ha} 
\end{gathered}
\eeq 
For $\phi \rightarrow \pi/2+$, we change variables to $z = -\cos(\phi)$ so that $z\rightarrow 0+$. Thus, $\sin \phi = (1 - z^2)^{\ha}$ and this is smooth near $z = 0$. Moreover, we have $w = \pm\left(-\frac{\cos(\phi)}{3\sin(\phi)}\right)^{1/2} = \pm\frac{z^{1/2}}{(3(1-z^2)^{1/2})^{1/2}}$ for $\pm t>0$, and this is smooth on $z>0$ near $z=0$ and continuous up to $z=0$. Hence, on $S_+$ we have the equation 
$
s =   H(z) z^{3/2}
$
for a function $H$ which is smooth on $z>0$ near $z=0$ and continuous up to $z=0$, with $H(0)\ne 0$. For $t < 0$ (i.e. on $S_-$), the exact same argument applies, except with the insertion of a minus sign since now $w=-(-3\tan\phi)^{-1/2}$. Hence, we get $s = -H(z)z^{3/2}$ on $S_-$.  Overall we have
\[s = H(z)z^{3/2}, z>0\text{ on }S_+\quad\text{and}\quad s = -H(z)z^{3/2}, z>0\text{ on }S_-,\]
so the two curves form a cusp at $s= 0, z = 0.$  

Moreover, we have
 \beq
 \begin{gathered}
 C\circ N^*\gamma = \{(s(t), \phi(t), \sigma(t), \eta(t)): s(t) = t\cos \phi(t) + t^3h(t) \sin\phi(t), \\
\phi(t) = \arctan(-1/(3t^2h(t)+t^3h'(t))), \eta(t) = \sigma(t)(-t\sin\phi(t) + (3t^2h(t)+t^3h'(t))\cos\phi(t) ) \}.
 \end{gathered}
 \eeq
For $\pm t>0$, the above set parametrizes $N^*S_{\pm}$. At $t=0$ we obtain the set
 \beq
 \{ (s,\phi,\sigma,\eta)\,:\,s = 0,  
\phi = \pi/2, \eta = 0\}.
 \eeq
Therefore, $C\circ N^*\gamma = \text{closure of } N^*S_+\cup N^*S_-.$ 
\end{proof}
We remark that for higher order inflection points, we would get a similar result that 
\beq
\begin{gathered}
s =   H(z) z^{k/2}, z>0\text{ on }S_+\quad\text{and} \quad s =   -H(z) z^{k/2}, z>0\text{ on }S_-,
\end{gathered}
\eeq
where $k\geq 3$ is odd. 
We also remark that another possible way to see that $S_+, S_-$ form a cusp is to look at the Lagrangian $C\circ N^*\gamma$ directly and apply Arnold's classification theorem \cite{Ar}.

Finally, we briefly point out  what happens when $\kappa(p_i) = 0$ but $p_i$ is not an inflection point. Then $\kappa(t)$ does not change sign across $t = 0.$ This happens when $\kappa_i = 0, i < k$ and $\kappa_k\neq 0$ for some $k$ even in \eqref{eq-curexp}. The simplest case is when $x_2(t) = t^4h(t),  h(0) \neq 0$.   Again, we assume that $h(t) > 0.$ 
Consider 
\[
\gamma = \{(x_1(t), x_2(t)) \in \mbr^2 : x_1(t) = t, x_2(t) = t^4h(t), h(t) > 0, t\in (-\delta, \delta)\}
\]
We denote by $\gamma_\pm$ the pieces of $\gamma$ where $\pm t > 0.$ Then 
$C\circ N^*\gamma_\pm = N^*S_\pm\backslash 0$ and 
 \beq
 \begin{gathered}
 S_\pm = \left\{[(s(t), \phi(t))]\in \mathcal{L}: s =  t\cos \phi + t^4 h(t) \sin\phi, 
  \phi = \arctan\left(-\frac{1}{4t^3 h(t) + t^4 h'(t)}\right), \pm t > 0\right\}
\end{gathered}
 \eeq
 At $t = 0$, we check that the two curves meet at $(s, \phi) = (0, \pi/2).$ 
Let $w =  t(h(t) + t h'(t)/4)^{\frac 13}, t\in (-\delta, \delta)$. Again, we see that $t = w g(w)$ on some $(-\eps, \eps)$ with $g(w)$ smooth and $g(0) \neq 0$. Now we have $\tan (\phi) = -1/(4w^3)$ so that 
\beq
w^3 = -\frac{\cos \phi}{4\sin \phi}
\eeq
In particular, for $\phi \neq \pi/2$, $w$ is a smooth function of $\phi$. We deduce that on $S_\pm$ 
\beq
\begin{gathered}
s = t \cos \phi + t^4 h(t) \sin \phi = wg(w)\cos\phi + w^4  g^4(w) h(wg(w))\sin\phi
\end{gathered}
\eeq
is a smooth function of $\phi$. The two curves $S_\pm$ meet tangentially at $\phi = \pi/2$ (note that the union of the curves forms a curve which is $C^1$ but not $C^2$). Although we do not pursue this case here, we remark that it can be addressed using the method we discuss in this paper. 

 \section{Commuting Radon transform with Pseudodifferential operators}\label{sec-commute}
 In this section, we aim to find the singularities in $\mcr f$ when $f \in I_k L^2(\mbr^2; \mcv(\gamma))$. There is a nice microlocal description of the space, which is due to the microlocal completeness of $\gamma$ introduced in \cite{MR1}. For any manifold $X$ and any Lagrangian submanifold $\Lambda\subset T^*X$, let 
\beq
\mcm(\Lambda) \defeq \{A \in \Psi^1(X): \sigma_1(A) = 0 \text{ on } \Lambda\backslash 0\}
\eeq
This is also a Lie algebra. In the case of $\Lambda = N^*Y$ where $Y\subset X$ is a smooth submanifold, we have for any $A\in \mcm(N^*Y)$ that 
$\sigma_1(A)$ is a $S^0(T^*X)$-linear combination of symbols of the form $\sigma_1(V)$ where $V\in\mcv(Y)$. In particular, 
\beq
\mcm(N^*Y) = \Psi^0(X) \mcv(Y) + \Psi^0(X).
\eeq
As such, in analogy to the vector field-based iterated regularity spaces defined in Section \ref{sec-lag}, we can consider iterated regularity spaces defined with respect to certain classes of pseudodifferential operators:
\beq
I_kH^s(X; \mcm(\Lambda)) \defeq \{u\in H^s(X): A_1A_2\cdots A_j u \in H^s(X), \forall j \leq k, A_j \in \mcm(\Lambda)\}. 
\eeq
This space then agrees with $I_kH^s(X;\mcv(Y))$ defined in Section \ref{sec-lag} when $\Lambda = N^*Y$.

In general, we will be interested in describing the regularity of distributions with respect to certain families of vector fields or pseudodifferential operators, so it is convenient to be able to move these operators across the Radon transform to describe the regularity of the Radon transform of a function in terms of the regularity of the function itself.

We will be interested in distributions $u$ on $\mathbb{R}^2$ which are compactly supported in some ball $B_r$; this corresponds to the region $s^2+(\eta/\sigma)^2<r^2$ in $T^*\mathcal{L}$. This means that $WF(\mathscr{R}u)$ is contained in this set; moreover singularities outside this set will be sent outside $B_r$ by the adjoint $\mathscr{R}^*$.

To move $\Psi$DOs across $\mcr$, we construct (approximate) parametrices for $\mcr$:
\begin{lemma}
There exists an FIO $Q^L$ which is a left parametrix for $\mcr$, i.e. such that $Q^L\mcr - I$ is a smoothing operator on $\mbr^2$. On the other hand, for every $r>0$ there exists an FIO $Q^R_r$ which gives a microlocal right parametrix for $\mcr$ in $\{s^2+(\eta/\sigma)^2<r^2\}$, i.e. with the property that $\mcr Q^R_r = I + E_r$ where $E_r\in\Psi^0(\mathcal{L})$ satisfies $WF'(E_r)\subset\{s^2+(\eta/\sigma)^2\ge r^2\}$.
\end{lemma}
The last statement means that the full symbol of $E_r$ (with respect to any quantization) decays rapidly outside the conic set $\{s^2+(\eta/\sigma)^2\ge r^2\}$, so that in particular if a distribution $u$ satisfies $WF(u)\subset\{s^2+(\eta/\sigma)^2<r^2\}$, then $WF(E_ru) \subset WF'(E_r)\cap WF(u) = \emptyset$, i.e. $E_ru$ is smooth.

\begin{proof}
It is well known (e.g. see \cite{Quinto}) that $\mathscr{R}^*\mathscr{R} = c|D|^{-1}$ for some constant $c$, i.e. it is the operator on $\mathbb{R}^2$ corresponding to the Fourier multiplier $c/|\xi|^{-1}$. Hence we have $c^{-1}|D|\mathscr{R}^*\mathscr{R} = \text{Id}$. Note that $|D|$ is not quite a $\Psi$DO on $\mathbb{R}^2$ since the symbol $|\xi|$ is singular at $\xi = 0$; however if we have $\chi_{\infty}\in C^{\infty}(\mathbb{R}^2)$ with $\chi_{\infty}\equiv 0$ on $|\xi|<1/2$ and $\chi_{\infty}\equiv 1$ on $|\xi|>1$, then $\chi_{\infty}(\xi)|\xi|$ is a smooth symbol and hence $\chi_{\infty}(D)|D|$ is a $\Psi$DO, with $|D|$ differing from $\chi_{\infty}(D)|D|$ by a smoothing operator. Hence $Q^L:=c^{-1}\chi_{\infty}(D)|D|\mathscr{R}^*$ is an FIO which inverts $\mathscr{R}$ modulo a smoothing operator.

For the other direction, there is a slight subtlety in doing the same procedure due to the fact that the projection $C\times C^*\cap\Delta_{T^*\mathbb{R}^2}\rightarrow (T^*\mathcal{L}\backslash 0)^2$ is not proper: indeed, in the canonical relation we have the identity $|x|^2 = s^2 + \left(\frac{\eta}{\sigma}\right)^2$, and for every $(s,\varphi;\sigma,\eta)$ with $\sigma\ne 0$ there exists $(x;\xi)$ with $((s,\varphi;\sigma,\eta),(x;\xi))\in C$; hence any nonempty neighborhood (even those with compact closure) of $(s_0,\varphi_0,0,\eta)$ (with $\eta\ne 0$) contains points where $\sigma\ne 0$ but $\sigma/\eta$ is arbitrarily small, so it has a preimage under the projection where $|x|$ is arbitrarily large. To fix this issue, we note that we are interested in distributions in $\mathbb{R}^2$ which are compactly supported, say in $B_r$ for some $r$. We now let $\chi_r\in C_c^{\infty}(\mathbb{R}^2)$ be identically $1$ in (a neighborhood of) $B_r$. We then consider the operator $\mathscr{R}\chi_r\mathscr{R}^*$: note that this is a well-defined operator $C^{\infty}(\mathcal{L})\rightarrow C_c^{\infty}(\mathcal{L})$ since $\chi_r\mathscr{R}^*: C^{\infty}(\mathcal{L})\rightarrow C_c^{\infty}(\mathbb{R}^2)$. Moreover, $\chi_r\mathscr{R}^*$ is a properly supported FIO, and $\mathscr{R}\chi_r\mathscr{R}^*$ is a $\Psi$DO of order $-1$ on $\mathcal{L}$ with symbol
\[(s,\varphi,\sigma,\eta)\mapsto\chi_r\left(s\theta_{\varphi}-\frac{\eta}{\sigma}\theta_{\varphi}^{\perp}\right)\times(\text{elliptic order }-1\text{ symbol}).\]
In particular it is elliptic on $\{s^2+(\eta/\sigma)^2<r^2\}$, so it can be microlocally inverted on this set, i.e. there exists $F_r\in\Psi^1(\mathcal{L})$ such that $\mathscr{R}\chi_r\mathscr{R}^*F_r = I + E_r$ where $E_r\in\Psi^0(\mathcal{L})$ and $WF'(E_r)\subset{\{s^2+(\eta/\sigma)^2\ge r^2\}}$. Thus $Q^R_r:=\chi_r\mathscr{R}^*F_r$ gives a right inverse for $\mathscr{R}$ up to an error microlocally supported in $\{s^2+(\eta/\sigma)^2\ge r^2\}$.
\end{proof}

Thus, we have the following:

\begin{prop}
\label{r2tol}
Suppose $A\in\Psi^m(\mathbb{R}^2)$ and $WF'(A)\subset\{(x,\xi)\,:\,|x|<r\}$ for some $r>0$. Then there exists $\tilde{A}\in\Psi^m(\mathcal{L})$ such that $\mathscr{R}A = \tilde{A}\mathscr{R}$ modulo a smoothing operator. Moreover, we have $\sigma_m(\tilde{A}) = \sigma_m(A)\circ C^{-1}$.
\end{prop}
In other words, we can always move $\Psi$DOs on $\mathbb{R}^2$ across the Radon transform.

\begin{proof}
Let $\chi\in C_c^{\infty}(\mathbb{R}^2)$ satisfy $\chi\equiv 1$ on $B_r$. Take $\tilde{A} = (\mathscr{R}\chi)(\chi A\chi)(\chi Q^L)$. Note that each of the terms is a properly supported FIO, and hence their composition is an FIO whose canonical relation is the diagonal relation on $T^*\mathcal{L}$, i.e. a $\Psi$DO. To calculate the principal symbol, we evaluate the principal symbol of each operator at the appropriate canonical relations to get
\begin{align*}\sigma(\tilde{A})(s,\phi,\sigma,\eta) &= \sigma(\mathscr{R}\chi)((s,\phi,\sigma,\eta),C^{-1}(s,\phi,\sigma,\eta)) \cdot\sigma(\chi A\chi)(C^{-1}(s,\phi,\sigma,\eta)) \\
&\cdot\sigma(\chi Q^L)(C^{-1}(s,\phi,\sigma,\eta),(s,\phi,\sigma,\eta)).
\end{align*}
Note that if $(x,\xi) = C^{-1}(s,\phi,\sigma,\eta)$ and $|x|\ge r$, then the middle term vanishes since $WF'(A)\subset\{|x|<r\}$. On the other hand, for $(s,\phi,\sigma,\eta)$ where the corresponding $(x,\xi)$ satisfies $|x|<r$, we then have $\chi(x) = 1$, and hence
\begin{align*} &\sigma(\mathscr{R}\chi)((s,\phi,\sigma,\eta),C^{-1}(s,\phi,\sigma,\eta))\cdot\sigma(\chi Q^L)(C^{-1}(s,\phi,\sigma,\eta),(s,\phi,\sigma,\eta)) \\
&=\sigma(\mathscr{R})((s,\phi,\sigma,\eta),C^{-1}(s,\phi,\sigma,\eta))\cdot\sigma(Q^L)(C^{-1}(s,\phi,\sigma,\eta),(s,\phi,\sigma,\eta)) \\
&=1
\end{align*}
since $Q^L\mcr = I$ up to smoothing. Hence $\sigma(\tilde{A}) = \sigma(\chi A\chi)\circ C^{-1} = \sigma(A)\circ C^{-1}$. Moreover, we have
\[\tilde{A}\mathscr{R} = \mathscr{R}AQ^L\mathscr{R} + \mathscr{R}(\chi^2A\chi^2-A)Q^L\mathscr{R} = \mathscr{R}A + \text{smoothing}\]
since $Q^L\mathscr{R} = I$ on $\mathbb{R}^2$ modulo smoothing, while $\chi^2A\chi^2-A\in\Psi^{-\infty}(\mathbb{R}^2)$ since $\chi\equiv 1$ on $WF'(A)$.
\end{proof}




As a corollary, we can always move $\Psi$DOs on $\mathbb{R}^2$ across the filtered backprojection as well. This implies that iterated regularity statements regarding distributions in the image of the filtered backprojection in physical space can be rephrased in terms of iterated regularity statements on the sinogram space.

\begin{cor}
\label{backprojcomm}
For the filtered backprojection $\mcr^*\mci^{-1} = c\mcr^*|D_s|$, for any $A\in\Psi^m(\mathbb{R}^2)$ with $WF'(A)\subset\{|x|<r\}$, there exists $\tilde{A}\in\Psi^m(\mathcal{L})$ such that $A\mcr^*\mci^{-1} = \mcr^*\mci^{-1}\tilde{A}$
modulo a smoothing operator. Moreover, we have $\sigma_m(\tilde{A}) = \sigma_m(A)\circ C^{-1}$.
\end{cor}

\begin{proof}
It suffices to solve for the adjoint $\tilde{A}^*$ in the adjoint equation $|D_s|\mcr A^* = \tilde{A}^*|D_s|\mcr$. By Proposition \ref{r2tol}, we can find $\tilde{B}\in\Psi^m(\mathcal{L})$ such that $\tilde{B}\mcr = \mcr A^*$ modulo smoothing, and moreover $\sigma_m(\tilde{B}) = \sigma_m(A^*)\circ C^{-1} = \sigma_m(A)\circ C^{-1}$. In particular, $WF(\tilde{B})\subset\{s^2+(\eta/\sigma)^2<r^2\}$. We would then like to let $\tilde{A}^* = |D_s|\tilde{B}|D_s|^{-1}$, since then $\tilde{A}^*|D_s|\mcr = |D_s|\tilde{B}\mcr = |D_s|\mcr A^*$ up to smoothing; however there is a subtle problem in that $|D_s|$ and $|D_s|^{-1}$ are not quite pseudodifferential operators on $\mathcal{L}$ (the ``symbols'' $|\sigma|^{\pm 1}$ fail to be symbolic in a conic neighborhood of $\sigma = 0$). Nonetheless, the non-symbolic behavior is away from where $\tilde{B}$ is microlocally supported. Hence, if $\chi(\sigma,\eta)$ (viewed as a function on $T^*\mathcal{L}$ independent of $(s,\phi)$) is supported in $\{|\eta/\sigma|<2r,|\sigma|>1\}$ and is identically $1$ on $\{|\eta/\sigma|<r,|\sigma|>2\}$ (note that this notion makes sense on $T^*\mathcal{L}$), then $\chi(\sigma,\eta)|\sigma|^{\pm 1}$ are symbols on $T^*\mathcal{L}$, and the corresponding $\Psi$DOs $\chi(D)|D_s|$ and $\chi(D)|D_s|^{-1}$ satisfy that $\chi(D)|D_s|\tilde{B}\chi(D)|D_s|^{-1}$ differs from $|D_s|\tilde{B}|D_s|^{-1}$ by a smoothing operator. Thus, letting $\tilde{A}^* = \chi(D)|D_s|\tilde{B}\chi(D)|D_s|^{-1}$, we have that $|D_s|\mcr A^* = \tilde{A}^*|D_s|\mcr$ modulo a smoothing operator, with $\sigma_m(\tilde{A}^*) = \sigma_m(B) = \sigma_m(A)\circ C^{-1}$; this gives $\sigma_m(\tilde{A}) = \sigma_m(A)\circ C^{-1}$ as well.
\end{proof}

For the other direction, we use the approximate right inverse $Q^{R}_r$ for $r$ sufficiently large:

\begin{prop}
\label{ltor2}
Suppose $\tilde{A}\in\Psi^m(\mathcal{L})$. Then for any $r>0$ there exists $A_r\in\Psi^m(\mathbb{R}^2)$ such that $\mathscr{R}A_r = \tilde{A}\mathscr{R} + \tilde{E}_r$ where 
$\tilde{E}_r:L_c^{\infty}(\mathbb{R}^2)\rightarrow\mathcal{D}'(\mathcal{L})$ satisfies $WF(\tilde{E}_ru)\subset\{s^2+(\eta/\sigma)^2\ge r^2\}\cap WF(\mathscr{R}u)$. Moreover $\sigma_m(A_r)|_{\{|x|<r\}} = \sigma_m(\tilde{A})\circ C|_{\{|x|<r\}}$, and the projection of $WF'(A_r)$ onto the base is compactly supported.
\end{prop}

\begin{proof}
Take $A_r = Q^R_r\tilde{A}\mathscr{R}$. Then $\mathscr{R}A_r = (\mathscr{R}Q^R_r)\tilde{A}\mathscr{R} = (I+E_r)\tilde{A}\mathscr{R}$. Furthermore, the symbol of $A_r$ is given by
\[\sigma(A_r)(x,\xi) = \sigma(Q^R_r)((x,\xi),C(x,\xi))\cdot\sigma(\tilde{A})(C(x,\xi))\cdot\sigma(\mcr)(C(x,\xi),(x,\xi)).\]
Since $Q^R_r\mcr - \text{Id}$ is microlocally trivial on $B_r$, we have $\sigma(Q^R_r)((x,\xi),C(x,\xi))\cdot\sigma(\mathscr{R})(C(x,\xi),(x,\xi)) = 1$ for $x\in B_r$, and hence the above expression just equals $\sigma(\tilde{A})(C(x,\xi))$ for $x\in B_r$. Note as well that for any $u$ we have $WF(E_r\tilde{A}\mathscr{R}u)\subset WF'(E_r)\cap WF(\mathscr{R}u)\subset{\{s^2+(\eta/\sigma)^2\ge r^2\}} \cap WF(\mcr u)$. Moreover, since $Q^R_r$ contains a term of the form $\chi_r(x)$, which is compactly supported, it follows that $A_r$ is microlocally trivial outside $\{(x,\xi)\,:\,x\in\text{supp }\chi_r\}$, and hence $WF'(A_r)$ projects into a compact region in the base.
\end{proof}




Now we are ready   to describe iterated regularity of $\mcr f$ in terms of iterated regularity of $f$:

\begin{lemma}
\label{iter-commute}
Suppose $\Lambda\subset T^*\mbr^2$ is Lagrangian, and let $\tilde\Lambda = C\circ\Lambda$ be the image of $\Lambda$ under the canonical relation $C$. Suppose $u\in I_kL^2(\mbr^2;\Lambda)$ is compactly supported. Then $\mcr u\in I_kH^{\ha}(\mathcal{L};\tilde\Lambda)$.
\end{lemma}

\begin{proof}
We want to consider the regularity of $\tilde{A}_1\tilde{A}_2\dots\tilde{A}_j\mcr u$ where $\tilde{A}_i\in\mcm(\tilde\Lambda)$. Suppose $\text{supp }u\subset B_r$. By Proposition \ref{ltor2}, we can find $A_i\in\Psi^1(\mbr^2)$ with $\mcr A_i = \tilde{A}_i\mcr + \tilde{E}_i$ where $WF(\tilde{E}_iv)\subset\{s^2+(\eta/\sigma)^2\ge r^2\}\cap WF(\mathscr{R}v)$ for all $v$. Since we also have $\sigma_1(A_i) = \sigma_1(\tilde{A}_i)\circ C$ and $\sigma_1(\tilde{A}_i)|_{C\circ\Lambda} = \sigma_1(\tilde{A}_i)|_{\tilde\Lambda} = 0$, it follows that in fact $A_i\in\mcm(\Lambda)$, and thus $A_{j'}A_{j'+1}\dots A_ju\in L^2(\mbr^2)$ for all $j'<j$. We have
\[\tilde{A}_j\mcr u = \mcr A_j u + \tilde{E}_ju,\quad WF'(\tilde{E}_ju)\subset \{s^2+(\eta/\sigma)^2\ge r^2\}\cap WF(\mathscr{R}u) = \emptyset\]
since $WF(\mathscr{R}u)\subset\{s^2+(\eta/\sigma)^2< r^2\}$ due to $\text{supp }u\subset B_r$. Thus, $\tilde{A}_j\mcr u = \mcr A_ju$ up to a \emph{smooth} error. Similarly, since $WF(A_{j'}A_{j'+1}\dots A_ju)\subset\{|x|<r\}$ for all $j'<j$, by induction we have $\tilde{A}_{j'}\tilde{A}_{j'+1}\dots\tilde{A_j}\mcr u = \mcr A_{j'}A_{j'+1}\dots A_ju$ up to a smooth error for all $1\le j'<j$. Taking $j' = 1$ gives the desired result upon noting that $\mcr:L^2(\mbr^2)\to H^{\ha}(\mathcal{L})$.
\end{proof}

In the case that $\Lambda = N^*\gamma$ where $\gamma = \partial D$ and $D$ is strictly convex, we have $C\circ N^*\gamma = N^*S\backslash 0$ is itself a conormal bundle of a smooth submanifold. Then $I_kH^{\ha}(\mathcal{L};N^*S) = I_kH^{\ha}(\mathcal{L};\mcv(S))$, so as a corollary we have:
\begin{cor}
\label{conv-ik}
If $\gamma = \partial D$ is the boundary of a strictly convex domain in $\mbr^2$ and $f\in {L^{\infty}I_kL^2(\mbr^2;\mcv(\gamma))}$ is compactly supported, then $\mcr f\in L^{\infty}I_kH^{\ha}(\mathcal{L};\mcv(S))$.
\end{cor}
The proof follows directly from Lemma \ref{iter-commute} along with the observation that $\mcr$ maps $L^{\infty}$ compactly supported functions on $\mbr^2$ to $L^{\infty}$ compactly supported functions on $\mathcal{L}$.

Next, consider the non-strictly convex case. We consider the model case near one inflection point $p_0$ on $\gamma$. We parametrize $\gamma$ locally near $p_0$ by $\gamma(t), t\in (-\delta, \delta), \delta > 0$ such that $p_0 = \gamma(0).$ Then we set $
\gamma_- = \gamma(t)|_{(-\delta, 0)}, \ \ \gamma_+ = \gamma(t)|_{(0, \delta)}.
$ 
Each $\gamma_\pm$ is strictly convex thus $\La_\pm = C\circ N^*\gamma_\pm = N^*S_\pm\backslash 0$. We know that $S_\pm$ form a cusp at $q_0$ the cusp point. Now, we need to introduce a Lagrangian distribution space for the cusp that allows us to analyze singularities of $\mcr f$ when $f$ is conormal to $\gamma.$   These spaces are introduced in Melrose \cite{Me}
 and  we follow the presentations in S\'a Barreto \cite[Section 3 and 5.2]{Sa}.

We consider the cusp $G = S_+\cup S_-$ in the sinogram. According to Arnold's result \cite{Ar}, it suffices to work with the model case that $G=\{(x, y)\in \mbr^2: x^2 = y^3\}$. Let $B = \{x = y = 0\}$ be the singular locus of $G$.  
The conormal bundle $\La_G = \text{closure of }\{N^*(G\backslash B)\}$ is a smooth closed Lagrangian submanifold and 
\beqq\label{eq-laG}
\La_G = \{(x, y, \xi, \eta)\in T^*\mbr^2: 4\eta^2 - 9y \xi^2 = 0, 3x\xi + 2y \eta = 0 \}
\eeqq
in which $(x, y, \xi, \eta)$ are the coordinates on $T^*\mbr^2$. 
We have the following result, again from Lemma \ref{iter-commute}:
 \begin{lemma}\label{lm-conormal2}
For $f \in  L^\infty I_k L^2(\mbr^2; \mcv(\gamma))$ supported near the inflection point $p_0$, we have $\mcr f\in L^\infty I_kH^{\ha}(\mbr^2; \mcm(\La_G))$.
 \end{lemma}
From these local descriptions, we can piece them together to describe a space to which $\mcr f$ (as well as $F(\mcr f)$ for a smooth function $F$) belongs if $f\in L^\infty I_kL^2(\mbr^2;\mcv(\gamma))$ where $\gamma$ satisfies assumptions (A1) and (A2). This will be done in the next section.

\section{The nonlinear interactions}\label{sec-non}
We analyze the singularities in $F(\mcr f)$ where $F$ is a smooth function. Our goal is to find a Lagrangian distribution space $\mca$ which contains both $\mcr f$ and $F(\mcr f)$. It suffices to find a space which contains $\mcr f$ and is also a $C^\infty$ algebra, so that then $F(\mcr f)$ belongs there as well.   
We recall that we assumed $\gamma = \cup_{i = 0}^N \overline \gamma_i$ in which $\gamma_i$ are disjoint open curves and strictly convex. The inflection points are denoted by $p_i, i = 1, 2, \cdots N. $ We  denote codimension one submanifolds $S_i \subset \mathcal{L}$ such that $C\circ N^*\gamma_i = N^*S_i\backslash 0. $   

When $f$ is conormal to $\gamma$, we know that $\mcr f$ is conormal to $S_i$ and  singular at the possible cusp points. To find the $C^\infty$ algebra $\mca$, we need to find out how the singular support of $\mcr f$ could possibly meet. We observe that it suffices to find $\mca$ locally because the action of the smooth function $F$ is local. Recall that we made the following assumption to avoid technical discussions related to multiple interactions: 
 any straight line $L \subset \mbr^2$ which is tangent to $\gamma$  is either tangent to (at most) two strictly convex pieces $\gamma_i, \gamma_j$, or tangent to $\gamma$ at one inflection point $p_i$.   

\subsection{The distribution space away from cusp points}
First of all, if $S_i, S_j$ intersect, they must intersect transversally which can be seen as follows. Suppose $S_i, S_j$ intersect at $p_{ij}$. Then by observation \eqref{tang-line}, the corresponding line $L$ in $\mbr^2$ must be tangent to both $\gamma_i$ and $\gamma_j$. Suppose that the intersection is not transversal and hence tangential. Then the (co)normal vectors of $S_i$ and $S_j$ are linearly dependent at $p_{ij}$, and in particular $N^*S_i\cap N^*S_j$ is a one-dimensional linear space. However, since the canonical relation $C$ is bijective outside the zero section, this can happen only if $N^*\gamma_i\cap N^*\gamma_j$ is one dimensional, which implies that $\gamma_i$ intersects $\gamma_j$ tangentially, contradicting our geometric assumption on $D$.


Now we look for the $C^\infty$ algebra near the intersection $S_i\cap S_j = \{p_{ij} \}$ which is a finite point set.  Locally, it suffices to consider the model case on $\mbr^2$ see \cite{MR1}, where $S_i = \{x = 0\}, S_j = \{y = 0\}$. Then the Lie algebra $\mcv(S_i, S_j)$ is spanned by $x\p_x, y\p_y$. We will consider the intersection of $I_kL^2(\mbr^2; \mcv(S_i, S_j))$  
with $L^\infty$, denoted by $L^\infty I_kL^2(\mbr^2; \mcv(S_i, S_j)) $. From  \cite{MR1}, we know that this is a $C^\infty$ algebra. Moreover, the space is microlocally complete, and we can decompose the space into
\beq
L^\infty I_kL^2(\mbr^2; \mcv(S_i, S_j)) = L^\infty I_kL^2(\mbr^2; \mcv(S_i, p_{ij}))  +  
L^\infty I_kL^2(\mbr^2; \mcv(S_j, p_{ij})).  
\eeq
Moreover, we note that
\[L^\infty I_kL^2(\mbr^2;\mcv(S_i,S_j))\supsetneq L^\infty I_kL^2(\mbr^2;\mcv(S_i))+L^\infty I_kL^2(\mbr^2;\mcv(S_j)).\]
Note that our X-ray data $\mcr f$ will locally belong to spaces of the form on the right-hand side, though nonlinear compositions need not return it there, which is why we are interested in using the space on the left-hand side. See \cite{PUW1} for an example of essentially similar behavior.


\begin{remark}
\label{sob-order}
Note that our data $\mcr f$ actually belongs to $H^{\ha}$-based iterated regularity spaces, so we are sacrificing Sobolev order in using the $L^2$-based spaces. Nonetheless, we still recover the qualitative aspects of the nonlinear interactions, so we will ignore the optimal Sobolev order in this paper.
\end{remark}

\subsection{The distribution space near cusp points}
When the cusp is involved, we cannot use the space $I_kL^2(\mbr^2; \mcm(\La_G))$ because  this is not an algebra as we show below.  
We shall use the {\em marked Lagrangian distribution} introduced by Melrose \cite{Me}, see also S\'a Barreto \cite{Sa}. 
 If $\La$ is a smooth closed embedded conic Lagrangian submanifold of $T^*\mbr^n\backslash 0$ and $\Sigma\subset \La$ is a smooth embedded conic hypersurface in $\La$, let 
\beq
\mcm(\La, \Sigma) = \{A\in \Psi^1(\mbr^n): \sigma_1(P) = 0 \text{ on $\La$ and $H_p$ is tangent to $\Sigma$}\}
\eeq
Then the marked Lagrangian distribution is defined as 
\beq
I_kL^2_\loc(\mbr^n, \mcm(\La, \Sigma)) = \{u\in L^2_\loc(\mbr^n): \mcm(\La, \Sigma)^j u\subset L^2_\loc, j\leq k.\}
\eeq
It suffices to consider  the model case that the cusp  $G =\{(x, y)\in \mbr^2: x^2 = y^3\}$. Let $B = \{x = y = 0\}$ be the singular locus of $G$. Let $\La_B = N^*B\backslash 0 =  \{(x, y, \xi, \eta): x = y = 0, (\xi, \eta)\neq 0\}$. Then $\La_G$ defined in \eqref{eq-laG} does not intersect $\La_B$ cleanly and 
\beq
\Sigma \defeq \La_G \cap \La_B = \{x = y = \eta = 0\}
\eeq
is a line in both $\La_G$ and $\La_B.$ We define
\beq
J_{k}(\mbr^2; G) \defeq I_k L^2_\loc(\mbr^2, \mcm(\La_G, \Sigma)) + I_k L^2_\loc(\mbr^2, \mcm(\La_B, \Sigma)).
\eeq
Using two blow-ups, Melrose \cite{Me} showed that this is a $C^\infty$ algebra. Here, we follow the non-homogeneous blow-ups introduced in   S\'a Barreto \cite[Section 3 and Section 5.2]{Sa}. 
First of all, we use $(x_1, x_2)$ for local coordinates on $\mbr^2$ near the cusp point $B = (0, 0)$. Then we blow up the cusp point using $x_1 = r w_1^2, x_2 = r w_2^3$ where $(w_1, w_2)$ is on 
\beq
S_{3-2}^1 \defeq \{(w_1, w_2): w_1^4 + w_2^6 = 1\}.
\eeq
We obtain the blown up space $X_{3-2} = S_{3-2}^1\times [0, \infty)$. Let $\beta: X_{3-2} \rightarrow \mbr^2$ be the blow down map. Then we let $G^{(1)} = \beta^{-1}(G)$ be the lift of the cusp to the blown-up space. Let $\mcw_G$ be the Lie algebra of vector fields in $X_{3-2}$ tangent to $G^{(1)}$ and to $\p X_{3-2}$. It is proved in  \cite[Theorem 5.2]{Sa} that 
\beq
\beta^*: J_k(\mbr^2; G) \rightarrow I_k L^2_{loc}(X_{3-2}, \mcw_G)
\eeq
is an isomorphism. Then it follows from Gagliardo-Nirenberg inequalities that the space is a $C^\infty$ algebra. 

S\'a Barreto showed in \cite{Sa} that $I_kL^2(\mbr^2; \mcm(\La_G)) \subset J_k(\mbr^2; G)$. However, it seems not known whether $L^\infty I_kL^2(\mbr^2; \mcm(\La_G))$ itself is a $C^\infty$ algebra. We clarify it below and leave the proof to the final section. This shows that the nonlinear interaction of cusp type singularities could produce new singularities, and it is necessary to work with a larger space. 
\begin{lemma}
\label{new-sing} 
For $k\ge 3$, there exist $u \in L^\infty I_kL^2(\mbr^2; \mcm(\La_G))$ such that $u^2\not\in I_kL^2(\mbr^2; \mcm(\La_G))$.
\end{lemma} 
The proof relies on an estimate of the Fourier transform of the square of a Lagrangian distribution associated to $\La_G$ first proven in \cite{Zw}; see Section \ref{sec-new} for more details.

Finally, we consider how cusps could possibly interact with other cusps or curves. Indeed, note that we cannot have two different cusps sharing a common cusp point, since by \eqref{tang-line} this would mean that the corresponding line is tangent to two different inflection points, which is ruled out by our assumption. Similarly, if a cusp $G = S_{i, +}\cup S_{i, -}$ intersects another curve $S_j$ at the cusp point, then there is a line tangent to $\gamma$ at the inflection point and a strictly convex piece of $\gamma$ which is also ruled out by our assumption.  

\subsection{The algebra $\mca$}
We now define our $C^{\infty}$ algebra as follows: let $\mca$ consist of all compactly supported $L^{\infty}(\mathcal{L})$ functions $g$ which satisfy the properties that
\[
\begin{gathered}
\text{if }\chi\in C_c^{\infty}(\mathcal{L})\text{ satisfies that }\text{supp }\chi\text{ is contained in some open set diffeomorphic to }\mbr^2\\
\text{and contains no cusp points and exactly one intersection point }p_{ij}\in S_i\cap S_j
\end{gathered}
\]
then $\chi g\in I_kL^2(\mbr^2; \mcv(S_i, S_j))$ (where we identify $\mbr^2$ with a neighborhood of $\text{supp }\chi$), while if instead
\[\begin{gathered}\text{supp }\chi\text{ is contained in some open set diffeomorphic to }\mbr^2,\text{contains no intersection points},\\
\text{and contains exactly one cusp point corresponding to the cusp }G
\end{gathered}\]
then $\chi g\in J_k(\mbr^2;G)$.

Note that $\mca$ really is a $C^{\infty}$ algebra. Indeed, for any $F\in C^{\infty}(\mathbb{R})$ and $g\in\mca$, if $\chi$ satisfies the first condition above, then there exists $\tilde\chi\in C_c^{\infty}(\mathcal{L})$ which satisfies the same support properties as $\chi$, but with the additional property that $\tilde\chi\equiv 1$ on $\text{supp }\chi$. We then have $\chi F(g) = \chi F(\tilde\chi g)$. Since $\tilde\chi g\in L^{\infty}I_kL^2(\mbr^2;\mcv(S_i,S_j))$ (by the support properties of $\tilde\chi$), and the latter is a $C^{\infty}$ algebra, it follows that $F(\tilde\chi g)$ is in $L^{\infty}I_kL^2(\mbr^2;\mcv(S_i,S_j))$ as well. A similar argument takes care of the case when $\chi$ satisfies the second condition.

The main use of this algebra is to provide a space to which $\mcr f$ belongs, and hence a space where the nonlinear composition $F(\mcr f)$ belongs as well:

\begin{lemma}
\label{art-alg}
If $f\in L^{\infty}I_kL^2(\mbr^2;\mcv(\gamma))$, then $\mcr f\in\mca$, and hence $F(\mcr f)\in\mca$ for any $C^{\infty}$ function $F$.
\end{lemma}

\begin{proof}
Recall the decomposition $\gamma = \cup{\overline{\gamma_i}}$ where each $\gamma_i$ is a strictly convex curve. Take a partition of unity $1 = \sum{\chi_i} + \sum{\chi_{ij}}$ on $\mbr^2$, where $\text{supp }\chi_{ij}$ is contained in a neighborhood of an inflection point in $\overline{\gamma_i}\cap\overline{\gamma_j}$, and $(\text{supp }\chi_i\cap\gamma)\subset\gamma_i$. Then $f = \sum{f_i} + \sum{f_{ij}}$ where $f_i = \chi_if$ and $f_{ij} = \chi_{ij}f$. We then have $f_i\in L^{\infty}I_kL^2(\mbr^2;\mcv(\gamma_i))$, and hence $\mcr f_i\in L^{\infty}I_kH^{1/2}(\mathcal{L};\mcv(S_i))\subset\mca$ by Lemma \ref{conv-ik}. Similarly, $f_{ij}$ is supported near an inflection point, so for the corresponding cusp point $G\in\mathcal{L}$, we have $\mcr f_{ij}\in L^{\infty}I_kH^{1/2}(\mathcal{L};\mcm(\La_G))\subset\mca$. It follows that $\mcr f = \sum{\mcr f_i} + \sum{\mcr f_{ij}}\in\mca$, and hence so is $F(\mcr f)$ by the $C^{\infty}$ algebra property of $\mca$.
\end{proof}
\begin{remark}
While, away from the cusp points, we can write $\mcr f$ as a sum of pieces each belonging to $L^{\infty}I_kH^{1/2}(\mathcal{L};\mcv(S_i))$, the same is not true once we apply our nonlinear composition $F$, mainly due to the interactions near the intersections $S_i\cap S_j$. This is one reason why we used the spaces $L^{\infty}I_kL^2(\mathcal{L};\mcv(S_i,S_j))$ in the construction of our $C^{\infty}$ algebra $\mca$; this is similar to the behavior investigated in \cite{PUW1}.
\end{remark}

\section{Characterization of the artifacts}\label{sec-char}
At last, we consider the term $\mcr^*\mci^{-1}F(\mcr f)$ when $F(\mcr f)\in \mca.$  It suffices to study the localized problem where $F(\mcr f)$ is supported either near a cusp   or near a transversal intersection. We continue using the notations in Section \ref{sec-non}. 

We first consider the local situation near a transversal intersection where $S_i$ and $S_j$ intersect at one point $p_{ij} = [(s_0,\phi_0)]\in\mathcal{L}$. Recall that the algebra $\mca$ locally resembles the space
\beq
L^\infty I_kL^2(\mbr^2; \mcv(S_i, S_j)) = L^\infty I_kL^2(\mbr^2; \mcv(S_i, p_{ij}))  +  
L^\infty I_kL^2(\mbr^2; \mcv(S_j, p_{ij}))  
\eeq
near $p_{ij}$. Note that by \eqref{can-rel} we have that for $(x,\xi)\in T^*\mbr^2\backslash 0$ we have
\[C(x,\xi)\in N^*\{[(s_0,\phi_0)]\}\iff s_0 = x\cdot\theta_{\phi_0}\text{ and }\xi\text{ is parallel to }\theta_{\phi_0}\text{, i.e. }(x,\xi)\in N^*\{x\cdot\theta_{\phi_0} = s_0\}.\]
Thus, $C^{-1}\circ N^*p_{ij} = N^*L$, where $L$ is a line tangent to $\gamma$ at $q_L, e_L \in \gamma$ which are non-inflection points.
We assumed that the tangency is of finite order. For such $L$, we consider a Lagrangian distribution space
 \beq
 \begin{gathered}
\mcb^s_k(\mbr^2; L) \defeq I_k H^s(\mbr^2; \mcv(L, \gamma))\\
= I_k H^s(\mbr^2; \mcv(\gamma, e_L)) + I_k H^s(\mbr^2; \mcv(L, e_L)) 
+ I_k H^s (\mbr^2; \mcv(\gamma, q_L)) + I_k H^s(\mbr^2; \mcv(L, q_L)) 
\end{gathered}
 \eeq
 It is possible to  modify the proof in \cite{MR1, Sa} and show that 
$L^\infty \cap \mcb_k^s $ is a $C^\infty$ algebra for suitable $s$. However, we will not pursue it here as it is not needed.

\begin{lemma}
\label{transverse-backproj}
Suppose $g\in L^{\infty}(\mathcal{L})$ is supported so that $\text{supp }g$ is compact, does not contain any cusp points, and contains exactly one intersection point $L\in\mathcal{L}$, corresponding to $S_i$ and $S_j$. If furthermore we have $g\in I_kL^2(\mathcal{L}; \mcv(S_i, S_j))$, then $\mcr^*\mci^{-1}g \in 
\mcb^{-\ha}_k(\mbr^2; L)$. 
\end{lemma}
\bpf
By the decomposition above, it suffices to show
\[u\in I_kL^2(\mathcal{L};\mcv(S_i,p_{ij}))\implies \mcr^*\mci^{-1}u\in I_kH^{-\ha}(\mbr^2,\mcv(\gamma_i,L))\]
since by \cite{MR1} the latter space equals $I_kH^{-\ha}(\mbr^2;\mcv(\gamma, e_L)) + I_kH^{-\ha}(\mbr^2;\mcv(L,e_L))$. 
By microlocal completeness \cite{MR1}, 
we have that 
\[I_kL^2(\mathcal{L};\mcv(S_i,p_{ij})) = I_kL^2(\mathcal{L};N^*(S_i,p_{ij})),\]
where the latter submanifold of $T^*\mathcal{L}$ is defined as
\[N^*(S_i,p_{ij}) := \{(s,\phi,\sigma,\eta)\in T^*\mathcal{L}\,:\,\sigma_1(V)(s,\phi,\sigma,\eta) = 0\text{ for all }V\in\mcv(S_i,p_{ij})\}\]
so in this case
\[N^*(S_i,p_{ij}) = N^*S_i\cup N^*p_{ij}.\]
Note that
\[C^{-1}(N^*S_i\cup N^*p_{ij}) = N^*\gamma_i\cup N^*L,\]
and again by microlocal completeness we have
\[I_kH^{-\ha}(\mathcal{L};\mcv(\gamma_i,L)) = I_kH^{-\ha}(\mathcal{L};N^*(\gamma_i,L)) = I_kH^{-\ha}(\mathcal{L};N^*\gamma_i\cup N^*L).\]
By Corollary \ref{backprojcomm}, for any $A\in\mcm(N^*\gamma_i\cup N^*L)$ we can find $\tilde{A}\in\Psi^1(M)$ such that $A\mcr^*\mci^{-1} = \mcr^*\mci^{-1}\tilde{A}$ up to smoothing error, with $\sigma_1(\tilde{A}) = \sigma_1(A)\circ C^{-1}$, so in particular $\sigma_1(\tilde{A})|_{N^*S_i\cup N^*p_{ij}} = \sigma_1(A)\circ C^{-1}|_{C(N^*\gamma_i\cup N^*L)}\equiv 0$, i.e. $\tilde{A}\in\mcm(N^*S_i\cup N^*p_{ij})$. Thus, if we wish to test $\mcr^*\mci^{-1}u$ against $A_1A_2\dots A_k$ with all $A_{l}\in\mcm(N^*\gamma_i\cup N^*L)$, we note that
\[A_1A_2\dots A_k\mcr^*\mci^{-1}u = \mcr^*\mci^{-1}\tilde{A}_1\tilde{A}_2\dots\tilde{A}_ku + \text{smooth} \in H^{-\ha}\]
since $\tilde{A}_1\tilde{A}_2\dots\tilde{A}_ku\in L^2$ as all $\tilde{A}_{l}\in\mcm(N^*(S_i,p_{ij}))$, and $\mcr^*\mci^{-1}$ maps from $L^2$ to $H^{-\ha}$.
\epf

Next, consider the situation near a cusp $G$ with cusp point $B$. Recall that locally $\mca$ resembles the algebra $L^{\infty}J_k(\mathcal{L};G)$ near $B$. Suppose $p$ is an inflection point and $C\circ N^*\gamma_\pm = N^*S_\pm$ and $S_\pm$ form a cusp $G$. The Lagrangian submanifolds involved near the cusp are $\La_G = N^*S_+\cup N^*S_-$ and $\La_B = N^*B$. We find that 
\beq
C^{-1}\circ \La_{G} = N^*\gamma_+\cup N^*\gamma_-
\eeq
Next, assume locally near $p$ that $\gamma_\pm = (t, t^3 h(t)), \pm t> 0$. Then $B = (0, \pi/2)$ and we find that 
\beq
\begin{gathered}
C^{-1}\circ\La_B = N^*(\{x\cdot\theta_{\pi/2} = 0\}) = N^*(\{x_2 = 0\}).
 \end{gathered}
\eeq
Note that the line $x_2 = 0$ which is the tangent line of $\gamma$ at $p$. We denote this tangent line by $L$. 
For this case, we consider conormal distribution space 
\beq
\mcc_k^s(\mbr^2; p) = I_kH^s(\mbr^2; \mcv(L, \gamma)).
\eeq

\begin{lemma}
\label{cusp-backproj}
If $g\in L^\infty(\mathcal{L})$ is supported so that $\text{supp }g$ contains no intersection points $p_{ij}$ and only contains one cusp point $G$, and furthermore $g\in J_k(\mathcal{L};G)$, then $\mcr^*\mci^{-1}g \in 
\mcc_k^{-\ha}(\mbr^2; p)$. 
\end{lemma}
\bpf
The proof is similar to that of Lemma \ref{transverse-backproj}. 
Consider $A \in \mcv(L, \gamma)$.  In particular, $\sigma_1(A)$ vanishes on $N^*L\cup N^*\gamma$.  
Applying Corollary \ref{backprojcomm}, we get $A \mcr^*\mci^{-1} = \mcr^* \mci^{-1} \tilde A$ modulo a smoothing operator,  where $\tilde A$ is a pseudo-differential operator of order $1$ on $\mathcal{L}$ with $\sigma(\tilde{A}) = \sigma(A)\circ C^{-1}$. In particular, $\sigma_1(\tilde A)$ vanishes on $C \circ N^*L = \La_B$ and $C\circ N^*\gamma = \La_G$. Since $\La_B$ and $\La_G$ are Lagrangian, it follows that the Hamilton vector field $H_{\sigma_1(\tilde A)}$ is tangent to both $\La_B$ and $\La_G$, thus tangent to $\La_B\cap \La_G = \Sigma$ which is the marking. Therefore, $\tilde A$ belongs to $\mcm(\La_B, \Sigma) + \mcm(\La_G, \Sigma)$. The rest of the proof is the same as in Lemma \ref{transverse-backproj}.
\epf

Finally, to conclude the general case from the local cases considered above,  we consider $\mcb_k^s = \sum_{L\in \mcl} \mcb_k^s(\mbr^2; L)$ where the summation is over all lines $L\in \mcl$ which is tangent at two non-inflection points, as well as $\mcc^s_{k} = \sum \mcc^s_k(\mbr^2; p)$ where the summation is over all inflection points on $\gamma.$ We also let $\mcd^s_k(\mcl) = \mcb^s_k + \mcc^s_k.$ We now show:
\begin{lemma}
\label{overall-backproj}
If $g\in\mca$, then $\mcr^*\mci^{-1}g\in\mcd^s_k(\mcl)$.
\end{lemma}
\begin{proof}
A partition of unity partitions $g = \sum{g_L}$ where each $g_L$ satisfies the assumptions of either Lemma \ref{transverse-backproj} or \ref{cusp-backproj}. Thus, $\mcr^*\mci^{-1}g_L\in \mcd^s_k(\mcl)$ for each $L$, and hence $\mcr^*\mci^{-1}g = \sum{\mcr^*\mci^{-1}g_L}\in\mcd^s_k(\mcl)$.
\end{proof}
\begin{proof}[Proof of Theorem \ref{thm-main1}]
By Lemma \ref{art-alg}, we have $F(\mcr f)\in\mca$, so by Lemma \ref{overall-backproj}, we have \\ $\mcr^*\mci^{-1}(F(\mcr f))\in\mcd^s_k(\mcl)$, as desired.
\end{proof}
 

\section{Generation of singularities at cusp points}\label{sec-new}

The study of nonlinear interaction of singularities has a rich history especially for nonlinear wave equations , see Beals \cite{Bea} for a review. Generation of new singularities has been considered in various scenarios, for example interaction of a cusp and plane in Zworski \cite{Zw},  interaction of swallowtails in Joshi-S\'a Barreto \cite{JoSa}.  Here, we study the cusp interactions and prove Lemma \ref{new-sing}. 

The main technical tool is the following lemma from \cite{Zw}:
\begin{lemma}[\cite{Zw}, Lemma 4]
\label{zw-lemma}
Suppose that $v\in\mathcal{S}'(\mbr^2)$ satisfies
\[\mathcal{F}(v)(\xi) = \chi(\xi_1)f(\xi_2/\xi_1)e^{-iH(\xi)}|\xi|^{-\rho}\]
where $H(\xi_1,\xi_2) = \xi_2^3/\xi_1^2$, $\chi$ is even and is identically zero on $\{|\xi|<1\}$ and identically $1$ on $\{|\xi|>2\}$, $f\in C_c^{\infty}(\mbr)$, and $\rho>2$. Then, for sufficiently small $\epsilon$ we have
\[\mathcal{F}(v^2)(\epsilon\xi_2,\xi_2) = c_{\pm}(\rho)(f(0)^2|\xi_2|^{-3\rho+5/2}+O(|\xi_2|^{-3\rho+3/2})),\]
with $c_{\pm}(\rho)\ne 0$.
\end{lemma}
Here we take the following convention for the Fourier transform:
\[\mathcal{F}(u)(\xi) = \int_{\mbr^n}{e^{-ix\cdot\xi}u(x)\,dx}.\]
The proof is by writing the above Fourier transform as a two-dimensional convolution integral and, after appropriate rescaling, applying stationary phase in one of the variables to get an integral of the form $c|\xi_2|^{-2\rho+3/2}\int_{\mathbb{R}}{|\tau|^{2\rho-3}F_0^{\epsilon}(\tau^2)e^{-i\xi_2\tau^2}\,d\tau}$ where $F_0^{\epsilon}$ is smooth with $F_0^{\epsilon}(0) = f(0)^2$, from which the above asymptotics follow. (Note that applying stationary phase to the entire integral would have resulted in no stationary points of the phase, but also would not have provided a lower bound for the derivative needed to apply stationary phase on non-compact regions.) We emphasize that the primary calculation in this section is provided in \cite{Zw}, and that the remainder of this section is using this result to prove Lemma \ref{new-sing}.

Note that for $v$ in the hypothesis of Lemma \ref{zw-lemma} we have
\begin{equation}
\label{cusp-conorm}
v(x) = (2\pi)^{-2}\int_{\mbr^2}{e^{i(x\cdot\xi-H(\xi))}b(\xi)\,d\xi},\quad b(\xi) = \chi(\xi_1)f(\xi_2/\xi_1)|\xi|^{-\rho}.
\end{equation}
It is easy to check that the phase function $(x,\xi)\mapsto x\cdot\xi - H(\xi)$ (away from $\xi_1 = 0$) parametrizes the Lagrangian
\[\Lambda = \{(x,\xi)\in T^*\mbr^2\backslash 0\,:\,3x_1\xi_1+2x_2\xi_2 = 0, 3\xi_2^2-x_2\xi_1^2 = 0\} = \overline{N^*\left(\left\{\left(\frac{x_1}{2}\right)^2=\left(\frac{x_2}{3}\right)^3\right\}\backslash 0\right)}\]
so that distributions of the form \eqref{cusp-conorm} are distributions conormal to the cusp $\{(x_1/2)^2 = (x_2/3)^3\}$. We will take $G$ to be the above cusp, with $B$ the cusp point $x=0$, and $\La_G$ to be the Lagrangian submanifold above.

Note that since $\rho>2$, the integrand in \eqref{cusp-conorm} is integrable, and hence $v\in L^\infty(\mbr^2)$. Thus overall we have $v\in L^\infty I_kL^2(\mbr^2;\mcm(\La_G))$ for all $k$. It thus suffices to show that the estimate of $\mathcal{F}(v^2)$ from Lemma \ref{zw-lemma} allows us to conclude that $v^2\not\in L^\infty I_kL^2(\mbr^2;\mcm(\La_G))$ for some $k$.

Note that $\La_G\cap\{x=0\} = \{x=0,\xi_2=0\}$, and if $u\in L^\infty I_kL^2(\mbr^2;\mcm(\La_G))$ for all $k$, then $u$ is a classical Lagrangian distribution associated to $\La_G$ and thus has wavefront set contained in $\La_G$. It thus suffices to show the following:
\begin{lemma}
\label{v2wf}
Let $v$ be in \eqref{cusp-conorm}. For sufficiently small $\epsilon>0$, we have $(0,(\epsilon,1))\in WF(v^2)$.
\end{lemma}
(By $(0,(\epsilon,1))$ we mean the point $x=0,\xi=(\epsilon,1)$ in $T^*\mbr^2$.) Indeed, this would imply that $WF(v^2)$ would not be contained in $\La_G$, and hence $v^2$ would not be in $I_kL^2(\mbr^2;\mcm(\La_G))$ for some $k$. To do so, we first show:

\begin{lemma}
\label{schwartz-infinity}
Let $v$ be in \eqref{cusp-conorm}. Then $v$ is Schwartz outside a compact region, i.e. there is a compact set outside of which $v$ agrees with a Schwartz function.
\end{lemma}
\begin{proof}
The main point to note is that $H'(\xi) = \left(-2\frac{\xi_2^3}{\xi_1^3},3\frac{\xi_2^2}{\xi_1^2}\right)$ is bounded on the support of $b$, since $b$ is supported in a region of the form $\left\{\left|\frac{\xi_2}{\xi_1}\right|<C\right\}$. Thus, suppose $|H'(\xi)|\le M$ on the support of $b$. The proof then proceeds by the usual integration by parts argument: letting $L$ be the differential operator
\[L = -i\frac{(x-H'(\xi))}{|x-H'(\xi)|^2}\cdot\partial_{\xi}\quad\text{so that}\quad Le^{i(x\cdot\xi-H(\xi))} = e^{i(x\cdot\xi-H(\xi))},\]
then for $|x|>2M$ (so that $|x-H'(\xi)|>\frac{1}{2}|x|$ on the support of $b$) we have that the coefficients of $L$ are bounded from above by $\frac{1}{|x-H'(\xi)|}\le \frac{2}{|x|}$, and moreover one can check that the coefficients' derivatives satisfy estimates of the form
\[\left|\partial_{\xi}^{\alpha}\left(\frac{x-H'(\xi)}{|x-H'(\xi)|^2}\right)\right|\le\frac{C_{\alpha}}{|x||\xi|^{|\alpha|}}\]
uniformly on $\{|x|>2M,\xi\in\text{supp }b\}$. Note that if $L^tg = i\partial_{\xi}\left(\frac{x-H'(\xi)}{|x-H'(\xi)|^2}\cdot\partial_{\xi}g\right)$ is the operator obtained by integration by parts against $L$, then we can write 
\begin{align*}\partial_x^{\alpha}v(x) = (2\pi)^{-2}\int_{\mbr^2}{e^{i(x\cdot\xi-H(\xi))}(i\xi)^{\alpha}b(\xi)\,d\xi} &= (2\pi)^{-2}\int_{\mbr^2}{L(e^{i(x\cdot\xi-H(\xi))})(i\xi)^{\alpha}b(\xi)\,d\xi} \\
&= (2\pi)^{-2}\int_{\mbr^2}{e^{i(x\cdot\xi-H(\xi))}(L^t)((i\xi)^{\alpha}b)(\xi)\,d\xi} \\
&=\dots = (2\pi)^{-2}\int_{\mbr^2}{e^{i(x\cdot\xi-H(\xi))}(L^t)^k((i\xi)^{\alpha}b)(\xi)\,d\xi}
\end{align*}
for any $k$ by repeatedly integrating by parts against the operator $L$. One can check that $(L^t)^k((i\xi)^{\alpha}b(\xi))$ is also supported in the support of $b$ (in particular in $|\xi_1|>1$) and satisfies estimates of the form
\[\left|(L^t)^k((i\xi)^{\alpha}b)(\xi)\right|\le C_k|x|^{-k}|\xi|^{-\rho+|\alpha|-k}\]
uniformly on $\{|x|>2M,\xi\in\text{supp }b\}$. For $k$ sufficiently large the integrand $e^{i(x\cdot\xi-H(\xi))}(L^t)^k((i\xi)^{\alpha}b)(\xi)$ becomes integrable, giving the estimate $|\partial_x^{\alpha}v(x)|\le C_k|x|^{-k}$ for all $k$ for $|x|>2M$. This shows that $v$ is Schwartz outside the ball $\{|x|\le 2M\}$.
\end{proof}

We next note that the estimate in Lemma \ref{zw-lemma} suggests that $(x,(\epsilon,1))$ is in $WF(v^2)$ for some $x$, given that the overall Fourier transform does not decay rapidly in the direction of $(\epsilon,1)$. This is true due to the contrapositive of the following plausible fact:

\begin{lemma}
\label{fourier-wf}
Suppose $u$ is a distribution on $\mathbb{R}^n$ which is Schwartz outside a compact region, and $\xi_0\in\mathbb{R}^n\backslash 0$ satisfies that $(x,\xi_0)\not\in\text{WF}(u)$ for all $x\in\mathbb{R}^n$. Then there is a conical neighborhood $\Gamma$ of $\xi_0$ in $\mathbb{R}^n$ such that the Fourier transform $\mathcal{F}(u)$ decays rapidly in $\Gamma$, i.e. that
\[\text{for all }N\text{ there exists }C_N\text{ such that }|\mathcal{F}(u)(\xi)|\le C_N(1+|\xi|)^{-N}\text{ for all }\xi\in\Gamma.\]
\end{lemma}
\begin{proof}
Suppose $u$ is smooth outside a compact region $K\subset\mathbb{R}^n$. Since $K\times\{\xi_0\}\subset\mathbb{R}^n\times\mathbb{R}^n=T^*\mathbb{R}^n$ is contained in the complement $WF(u)^C$ of $WF(u)$ by hypothesis, and $WF(u)^C$ is open and conical, it follows that there is a conical neighborhood $\Gamma'$ of $\xi_0$ such that $K\times\Gamma'\subset\text{WF}(u)^C$; note then that $\mathbb{R}^n\times\Gamma'\subset\text{WF}(u)^C$. Let $A$ be a Fourier multiplier operator where the Fourier multiplier $a(\xi)$ is a symbol supported in $\Gamma'$ and identically one on an open conical neighborhood $\Gamma$ of $\xi_0$; note then that $A$ is an order $0$ pseudodifferential operator. Let $\chi\in C_c^{\infty}(\mathbb{R}^n)$ be identically $1$ on a ball large enough to contain $\text{sing supp }u$, and let $\tilde\chi\in C_c^{\infty}(\mathbb{R}^n)$ be identically $1$ on $\text{supp }\chi$. We write
\[Au = \tilde\chi A\chi u + (1-\tilde\chi)A\chi u + A(1-\chi)u.\]
We claim that $Au$ is Schwartz, and to do so we show each of the terms on the right-hand side is Schwartz. Note that by construction we have $\text{WF}'(A)\cap WF(u) = \emptyset$, and hence $WF(\tilde\chi A\chi u)\subset WF'(\tilde\chi A)\cap WF(\chi u)\subset WF'(A)\cap WF(u) = \emptyset$, so $\tilde\chi A\chi u$ is smooth. Moreover it is compactly supported by construction, and hence $\tilde\chi A\chi u$ is Schwartz. The operator $(1-\tilde\chi)A\chi$ has Schwartz kernel supported away from the diagonal due to the support properties of $\chi$ and $\tilde\chi$, and hence $(1-\tilde\chi)A\chi$ is an order $-\infty$ pseudodifferential operator; in particular it maps $u$ to a Schwartz function, i.e. $(1-\tilde\chi)A\chi u$ is Schwartz. Finally, since $u$ is Schwartz at infinity, and $(1-\chi)u$ is smooth due to $\chi$ being identically $1$ on $\text{sing supp }u$, it follows that $(1-\chi)u$ is Schwartz, and the pseudodifferential operator $A$ will preserve that property, i.e. $A(1-\chi)u$ is Schwartz as well. Thus, $Au$ is Schwartz, and so is its Fourier transform $a(\xi)\mathcal{F}(u)(\xi)$; the fact that $a$ is identically $1$ on $\Gamma$ thus gives the conclusion.
\end{proof}

\begin{proof}[Proof of Lemma \ref{v2wf}]
By Lemma \ref{zw-lemma}, we know that $\mathcal{F}(v^2)$ does not decay rapidly in a conical neighborhood of $\xi_0 = (\epsilon,1)$ for all sufficiently small $\epsilon>0$. Since $v^2$ is Schwartz outside a compact region by Lemma \ref{schwartz-infinity}, we use the contrapositive of Lemma \ref{fourier-wf} to conclude that $(x,(\epsilon,1))\in WF(v^2)$ for some $x\in\mbr^2$. We now show that this $x$ must be the origin, i.e. the cusp point $B$. Indeed, away from $B$ the cusp $G$ is a union of smooth curves, and hence $I_kL^2(\mbr^2\backslash B;\mcm(\La_G\backslash B)) = I_kL^2(\mbr^2\backslash B;\mcv(G\backslash B))$; in particular it is closed under multiplication. Thus away from $B$ we have that $v^2$ is locally in $I_kL^2(\mbr^2\backslash B;\mcm(\La_G\backslash B))$ (more precisely $\chi v^2\in I_kL^2(\mbr^2\backslash B;\mcm(\La_G\backslash B))$ for all $k$ for any $\chi\in C_c^{\infty}(\mbr^2\backslash B)$), so its wavefront set away from $B$ is contained in $\La_G$. Since
\[\La_G = \left\{(x,\xi)\in T^*\mbr^2\,:\, x = H'(\xi) = \left(-2\frac{\xi_2^3}{\xi_1^3},3\frac{\xi_2^2}{\xi_1^2}\right)\right\},\]
it follows that
\[(x,(\epsilon,1))\in \La_G\implies x = (-2\epsilon^{-3},3\epsilon^{-2}).\]
However, by Lemma \ref{schwartz-infinity} we know that $v^2$ is Schwartz away from a compact region. Hence, for $\epsilon$ sufficiently small, we have $(x,(\epsilon,1))\not\in WF(v^2)$ for any $x\ne 0$, implying that $(0,(\epsilon,1))\in WF(v^2)$ as desired.
\end{proof}

\begin{proof}[Proof of Lemma \ref{new-sing}]
For $v$ as in \eqref{cusp-conorm}, we have $v\in L^{\infty}I_kL^2(\mbr^2;\mcm(\La_G))$ for all $k$ since it is a classical Lagrangian distribution associated to $\La_G$, but by Lemma \ref{v2wf}, we have that $WF(v^2)\not\subset \La_G$, since $WF(v^2)$ contains $(0,(\epsilon,1))\not\in\La_G$. It follows that $v^2$ cannot be a classical Lagrangian distribution associated to $\La_G$, and hence there must exist $k$ for which $v^2\not\in I_kL^2(\mbr^2;\mcm(\La_G))$.
\end{proof}

\begin{remark}
We can take our example to be of the form $\mcr u$ for a function $u$ conormal to a curve in $\mbr^2$ with an inflection point: indeed, if $u(x) = g(x_1)h(x_2-x_1^3)$, where $g\in C_c^{\infty}(\mbr)$, $h:\mathbb{R}\to\mathbb{R}$, and $h(y) = (2\pi)^{-1}\int_{\mathbb{R}}{e^{iy\eta}a(\eta)\,d\eta}$ where $a$ is a symbol on $\mathbb{R}$, then $u$ will be conormal to the curve $x_2 = x_1^3$, which contains a simple inflection point at $(0,0)$. Then the Radon transform can be written as
\begin{align*} \mathscr{R}u(s,\phi) &= \int_{\mbr^2}{e^{i\lambda(s-x_1\cos\phi-x_2\sin\phi)}g(x_1)h(x_2-x_1^3)\,d\lambda\,dx}\\
&=\frac{1}{\sin\phi}\int_{\mbr^2}{e^{i\lambda'(z+x_1w-x_2)}g(x_1)h(x_2-x_1^3)\,d\lambda'\,dx}
\end{align*}
where $\lambda=\lambda'/\sin\phi$, $z = s/\sin\phi$, and $w = -\cot\phi$ (note that the change of coordinates $(s,\phi)\mapsto(z,w)$ is a smooth change of coordinates near $\phi = \pi/2$). If we now let $\xi_1 = \lambda'$ and $\xi_2 = \lambda x_1$, then the integral becomes
\[\int{e^{i(\xi_1z+\xi_2w)}e^{-i\xi_1x_2}g(\xi_2/\xi_1)h(x_2 - \xi_2^3/\xi_1^3)\,d\xi_1\,d\xi_2\,dx_2}.\]
We now let $y = x_2-\xi_2^3/\xi_1^3$, in which case the integral becomes
\[\int{e^{i\left(\xi_1z+\xi_2w-\frac{\xi_2^3}{\xi_1^2}\right)}g(\xi_2/\xi_1)\left(\int{e^{-i\xi_1y}h(y)\,dy}\right)\,d\xi} = \int{e^{i\left(\xi_1z+\xi_2w-\frac{\xi_2^3}{\xi_1^2}\right)}g(\xi_2/\xi_1)a(\xi_1)\,d\xi}.\]
If we let $v$ be a function defined locally near $(z,w) = (0,0)$ to match $\mcr u(s,\phi)$ under the change of coordinates above, then $v$ is a distribution of the form \eqref{cusp-conorm}, as long as $a(\eta)=\chi(\eta)|\eta|^{-\rho}$ where $\chi$ satisfies the hypothesis in Lemma \ref{zw-lemma}.
\end{remark}

\section*{Acknowledgements}
The authors thank Andr\'as Vasy for useful discussions regarding commuting the Radon transform and Ant\^onio S\'a Barreto for helpful conversations on cusp singularities. The authors also acknowledge the Institute for Mathematics and its Applications (IMA) in Minneapolis, Minnesota for hosting the workshop ``Mathematics in Optical Imaging'' where the authors began discussing this project, as well as the Mathematical Sciences Research Institute (MSRI) in Berkeley, California for hosting the workshop ``Recent developments in microlocal analysis'' where the authors continued to collaborate on this project. The authors also thank the referees for their helpful comments in improving the exposition of this paper.


\bibliography{metal}

\begin{thebibliography}{10}

\bibitem{Ar}
V.~I. Arnold.
\newblock Wave front evolution and equivariant {M}orse lemma.
\newblock {\em Communications on Pure and Applied Mathematics}, 29(6):557--582,
  November 1976.

\bibitem{Bea}
M.~Beals.
\newblock {\em Propagation and interaction of singularities in nonlinear
  hyperbolic problems}, volume 130.
\newblock Springer Science \& Business Media, 2012.

\bibitem{BFJQ}
L.~Borg, J.~Frikel, J.~Jorgensen, and E.~T. Quinto.
\newblock Analyzing reconstruction artifacts from arbitrary incomplete {X}-ray
  {CT} data.
\newblock {\em SIAM Journal on Imaging Sciences}, 11(4):2786--2814, 2018.

\bibitem{Gu}
V.~Guillemin.
\newblock On some results of {G}elfand in integral geometry.
\newblock {\em Proc. Symp. Pure Math.}, 43, 1985.

\bibitem{Ho3}
L.~H{\"o}rmander.
\newblock {\em The analysis of linear partial differential operators {III}:
  Pseudo-differential operators}.
\newblock Classics in Mathematics, 2007.

\bibitem{JoSa}
M.~Joshi and A.~{S{\'a} Barreto}.
\newblock The generation of semilinear singularities by a swallowtail caustic.
\newblock {\em Amer. J. Math.}, 120:529--550, 1998.

\bibitem{Kat}
A.~Katsevich.
\newblock Local tomography with nonsmooth attenuation.
\newblock {\em Transactions of the American Mathematical Society},
  351(5):1947--1974, 1999.

\bibitem{Kl}
W.~Klingenberg.
\newblock {\em A course in differential geometry}, volume~51.
\newblock Springer Science \& Business Media, 2013.

\bibitem{Me}
R.~Melrose.
\newblock Semilinear waves with cusp singularities.
\newblock {\em Journ{\'e}es {\'E}quations aux Deriv{\'e}es Partielles}, 1987.

\bibitem{MR1}
R.~Melrose and N.~Ritter.
\newblock Interaction of nonlinear progressing waves for semilinear wave
  equations.
\newblock {\em Annals of Mathematics}, pages 187--213, 1985.

\bibitem{MR2}
R.~Melrose and N.~Ritter.
\newblock Interaction of progressing waves for semilinear wave equations. {II}.
\newblock {\em Arkiv f{\"o}r Matematik}, 25(1):91--114, 1987.

\bibitem{PUW1}
B.~Palacios, G.~Uhlmann, and Y.~Wang.
\newblock Quantitative analysis of metal artifacts in {X}-ray tomography.
\newblock {\em SIAM Journal on Mathematical Analysis}, 50(5):4914--4936, 2018.

\bibitem{Seo}
H.~S. Park, J.~K. Choi, and J.~K. Seo.
\newblock Characterization of metal artifacts in {X}-ray computed tomography.
\newblock {\em Communications on Pure and Applied Mathematics},
  70(11):2191--2217, 2017.

\bibitem{Seo0}
H.~S. Park, Y.~E. Chung, and J.~K. Seo.
\newblock Computed tomographic beam-hardening artefacts: mathematical
  characterization and analysis.
\newblock {\em Philosophical Transactions of the Royal Society A: Mathematical,
  Physical and Engineering Sciences}, 373(2043), 2015.

\bibitem{Quinto}
E.~T. Quinto.
\newblock {\em An introduction to X-ray tomography and Radon transforms},
  volume~63, chapter~1, pages 1--23.
\newblock Proc. Sympos. Appl. Math., 2006.

\bibitem{Sa}
A.~{S{\'a} Barreto}.
\newblock Second microlocal ellipticity and propagation of conormality for
  semilinear wave equations.
\newblock {\em Journal of Functional Analysis}, 102(1):47--71, 1991.

\bibitem{Zw}
M.~Zworski.
\newblock An example of new singularities in the semi-linear interaction of a
  cusp and a plane.
\newblock {\em Communications in Partial Differential Equations}, 19:901--909,
  1994.

\end{thebibliography}
\bibliographystyle{plain}

\end{document}